\documentclass[12pt,reqno]{amsart}
\usepackage{amsmath,amssymb,amsfonts,amsthm}
\usepackage[left=3cm, right=3cm, top=2.8cm, bottom=2.8cm]{geometry}
\usepackage[utf8]{inputenc}
\usepackage[T1]{fontenc}
\usepackage{xcolor}
\usepackage{graphicx}
\usepackage{hyperref}
\hypersetup{
    colorlinks=true,
    linkcolor=blue,
    filecolor=magenta,
    urlcolor=cyan,
}

\newtheorem{theorem}{Theorem}
\newtheorem{lemma}[theorem]{Lemma}
\newtheorem{proposition}[theorem]{Proposition}
\newtheorem{corollary}[theorem]{Corollary}
\newtheorem{definition}[theorem]{Definition}
\newtheorem{remark}[theorem]{Remark}

\theoremstyle{plain}

\begin{document}

\title[Energy Estimates for Fractional Evolution Equations]{Energy Estimates for Fractional Evolution Equations}


\author[P. M. Carvalho-Neto]{Paulo M. Carvalho-Neto}
\address[Paulo M. Carvalho-Neto]{Department of Mathematics, Federal University of Santa Catarina, Florian\'{o}polis - SC, Brazil\vspace*{0.3cm}}
\email[]{paulo.carvalho@ufsc.br}

\author[C. L. Frota]{C\'{i}cero L. Frota}
\address[C\'{i}cero L. Frota]{Department of Mathematics, State University of Maringá, Maring\'{a} - PR, Brazil\vspace*{0.3cm}}
\email[]{clfrota@uem.br}

\author[J. C. O. Ballesteros]{Juan C. Oyola Ballesteros}
\address[Juan C. Oyola Ballesteros]{Department of Mathematics, Federal University of Santa Catarina, Florian\'{o}polis - SC, Brazil\vspace*{0.3cm}}
\email[]{jcoyolaba@unal.edu.co}

\author[P. G. P. Torelli]{Pedro G. P. Torelli}
\address[Pedro G. P. Torelli]{Department of Mathematics, State University of Maringá, Maring\'{a} - PR, Brazil\vspace*{0.3cm}}
\email[]{pgptorelli@gmail.com}






\begin{abstract}
This work establishes a more applicable version of an energy inequality for weak solutions of evolution equations involving fractional time derivatives. Unlike the classical identity, $(d/dt)\|u(t)\|_H^2 = 2(u'(t), u(t))_H$, its fractional counterpart requires a regularity condition that is often difficult to verify in practice.

We revisit the inequality in the setting $H = L^2(\Omega)$, where $\Omega \subset \mathbb{R}^n$ is open, and present a new proof under the sole assumption that $u \in L^\infty(0,T;H)$. In our approach, the necessary regularity arises naturally during the argument, rather than being assumed a priori. This refinement extends the range of applicability of the inequality and provides a rigorous basis for using the Faedo–Galerkin method in fractional evolution problems. As an application, we establish the existence, uniqueness, and regularity of weak solutions to the fractional heat equation.
\end{abstract}

\maketitle

\section{Introduction} 

One of the most fundamental tools in the analysis of weak solutions to evolution equations is the identity
\begin{equation}\label{202505271005}
\dfrac{d}{dt}\|u(t)\|^2_H = 2(u'(t), u(t))_H,\quad \text{for a.e. } t \in [0, T],
\end{equation}
valid for functions $u \in W^{1,2}(0, T; H)$, where $H$ is a Hilbert space. This relation plays a central role in deriving energy estimates and a priori bounds in the classical theory of partial differential equations.

In particular, identity~\eqref{202505271005} is a key ingredient in the Faedo-Galerkin meth\-od, a classical technique for proving existence of weak solutions by approximating the original problem with a sequence of finite-dimensional ones. The energy identity enables the derivation of uniform bounds for the Galerkin approximations, which are instrumental in extracting weak and weak-* convergent subsequences.

\subsection{The Fractional Version of (\ref{202505271005}) and the Main Result}
When dealing with evolution equations involving Riemann–Liouville (i.e., $D_t^\alpha$) or Caputo (i.e., $cD_t^\alpha$) fractional derivatives of order $\alpha \in (0,1)$, it is well known that the identity \eqref{202505271005} no longer holds; see \cite{CarFe0} for a detailed discussion. Nevertheless, within the same work, the second author and R. Fehlberg Júnior established a fractional analogue of a generalized version of this identity (cf. \cite[Chapter 3, Lemma 1.2]{Te1}), under certain additional regularity assumptions:
\begin{theorem}[{\cite[Theorem 4.16]{CarFe0}}]\label{202507181644}
Let $H$ and $V$ be Hilbert spaces satisfying:
\begin{itemize}
    \item[(a)] $V$ is dense in $H$ and continuously embedded into it;
    \item[(b)] $H \equiv H'$ via the Riesz representation theorem, where $H'$ denotes the dual of $H$.
\end{itemize}

Under these conditions, if $V'$ denotes the dual of $V$, then we have the continuous embeddings
\[
V \subset H \equiv H' \subset V'.
\]

\begin{itemize}
\item[(i)] Suppose that $p > 1/(1-\alpha)$ and $p \geq 2$. If $u \in L^p(t_0,t_1;V)$ satisfies $D_{t_0,t}^\alpha u \in L^2(t_0,t_1;V')$ and $J_{t_0,t}^{1-\alpha}\|u(t)\|_H^2 \in W^{1,1}(t_0,t_1;\mathbb{R})$, then $u$ is almost everywhere equal to a continuous function from $[t_0,t_1]$ into $H$, and
\[
D_{t_0,t}^\alpha \|u(t)\|_H^2 \leq 2 \langle D_{t_0,t}^\alpha u(t), u(t) \rangle_{V',V}, \quad \text{for a.e. } t \in [t_0,t_1].
\]

\item[(ii)] If $u \in L^2(t_0,t_1;V)$, $cD_{t_0,t}^\alpha u \in L^2(t_0,t_1;V')$, and $J_{t_0,t}^{1-\alpha}\|u(t)\|_H^2 \in W^{1,1}(t_0,t_1;\mathbb{R})$, then $u$ is almost everywhere equal to a continuous function from $[t_0,t_1]$ into $H$, and
\[
cD_{t_0,t}^\alpha \|u(t)\|_H^2 \leq 2 \langle cD_{t_0,t}^\alpha u(t), u(t) \rangle_{V',V}, \quad \text{for a.e. } t \in [t_0,t_1].
\]
\end{itemize}
\end{theorem}

In the particular case where $t_0 = 0$, $t_1 = T$, $H = V$, and either $D_{0,t}^\alpha u \in L^2(0,T; H)$ or $cD_{0,t}^\alpha u \in L^2(0,T; H)$, the duality pairing reduces to the inner product in $H$, and the inequalities above become
\begin{equation}\label{202505271019}
D_t^\alpha \|u(t)\|_H^2 \leq 2 (D_t^\alpha u(t), u(t))_H, \quad \text{a.e. } t \in [0, T],
\end{equation}
or
\begin{equation}\label{202505271020}
cD_t^\alpha \|u(t)\|_H^2 \leq 2 (cD_t^\alpha u(t), u(t))_H, \quad \text{a.e. } t \in [0, T],
\end{equation}
depending on whether one adopts the Riemann--Liouville or Caputo definition of the fractional derivative. Both expressions can be viewed as fractional analogues of the classical identity \eqref{202505271005}, which is fundamental to energy inequalities in the standard theory.

Motivated by this technical obstacle, we present a new proof of \eqref{202505271019} and \eqref{202505271020} in the setting $H = L^2(\Omega)$, where $\Omega \subset \mathbb{R}^n$ is open, under the additional assumption that $u \in L^\infty(0,T;L^2(\Omega))$. 

A key feature of our argument is that the regularity $J_{t_0,t}^{1-\alpha}\|u(t)\|_H^2 \in W^{1,1}(t_0,t_1;\mathbb{R})$ follows directly from the properties of the function $u$, rather than being assumed a priori. This refinement simplifies the assumptions required to apply the fractional energy inequality, broadens its applicability, and provides a rigorous foundation for employing the Faedo–Galerkin method in fractional evolution problems.

Our contribution advances the theory of fractional differential equations and paves the way for future developments, such as the study of fractional heat equations, as explored in the final section, and further topics to be addressed in ongoing work.

\subsection{Outline of the Paper}
The paper is organized as follows. Section \ref{pre} recalls some classical concepts and introduces several technical identities and auxiliary results that are used throughout the manuscript. In Section \ref{maintheorem}, we prove our main results, namely Theorem \ref{teodesdrl} and Corollary \ref{teodesdcap}. Finally, Section \ref{fracheat} illustrates how these results can be applied to establish existence, uniqueness, and regularity for a fractional heat equation in open and bounded subsets of $\mathbb{R}^n$.

\section{Notation and Preliminaries}
\label{pre}

We begin by recalling some classical function spaces, focusing on the Bochner--Lebesgue and Bochner--Sobolev spaces, which generalize the usual $L^p$ and $W^{1,p}$ spaces to the setting of Banach space-valued functions. Throughout this section, we assume that $X$ is a Banach space and that $T > 0$ is fixed. Integrability and differentiability for functions $f: [0,T] \to X$ are understood in the Bochner sense; for details, see~\cite{ArBaHiNe1,CaHa1,DuSc1}.

\begin{itemize}
\item[(i)] For $1 \leq p \leq \infty$, the space $L^p(0,T;X)$ consists of all Bochner-measurable functions $f: [0,T] \to X$ such that $\|f(t)\|_X \in L^p(0,T)$. Endowed with the norm
$$
\|f\|_{L^p(0,T;X)} := \left( \int_0^T \|f(t)\|_X^p \, dt \right)^{1/p}, \quad \text{for } 1 \leq p < \infty,
$$
and
$$
\|f\|_{L^\infty(0,T;X)} := \operatorname*{ess\,sup}_{t \in [0,T]} \|f(t)\|_X,
$$
this space is a Banach space.

\item[(ii)] For $1 \leq p \leq \infty$, the Sobolev space $W^{1,p}(0,T;X)$ consists of all functions $f \in L^p(0,T;X)$ whose weak derivative $f'$ exists and belongs to $L^p(0,T;X)$. Equipped with the norm
$$
\|f\|_{W^{1,p}(0,T;X)} := \|f\|_{L^p(0,T;X)} + \|f'\|_{L^p(0,T;X)},
$$
this is also a Banach space.

\item[(iii)] For $k \in \mathbb{Z}_+:=\{0,1,\ldots\}$, we denote by $C^k([0,T];X)$ the Banach space of functions $f: [0,T] \to X$ that are $k$ times continuously differentiable, equipped with the norm
$$
\|f\|_{C^k([0,T];X)} := \sum_{j=0}^k \sup_{t \in [0,T]} \|f^{(j)}(t)\|_X.
$$
In this context, $f^{(j)}$ denotes the $j$-th classical (i.e., pointwise) derivative of $f$. When $k = 0$, we can also write $C([0,T];X)$ instead of $C^0([0,T];X)$. 

\item[(iv)] We denote by $C^\infty([0,T];X)$ the space of functions that are infinitely differentiable from $[0,T]$ into $X$, that is,
$$
C^\infty([0,T];X) := \bigcap_{k=0}^\infty C^k([0,T];X).
$$

\item[(v)] Recall that the support of a function $\varphi: (0,T) \to X$ is defined as
$$
\operatorname{supp}(\varphi) := \overline{\{ t \in (0,T) \mid \varphi(t) \neq 0 \}}.
$$
Finally, the space of test functions, denoted by $C_c^\infty((0,T);X)$, is the space of all functions $\varphi \in C^\infty([0,T];X)$ such that $\operatorname{supp}(\varphi)$ is compactly contained in $(0,T)$. 

\end{itemize}

The notions introduced above are sufficient for us to define the fractional integral and derivatives considered in the present work. These concepts have been widely employed in the study of fractional calculus (see, for instance, \cite{AlCaVa1,CarFe0,CarFe1,CarFe2,CarFe3,CarFe4,CarFe05,CarPl1,GiNa1,LiPeJi1} as few examples).

\begin{definition}
Let $0 < \beta < \infty$ and $f : [0,T] \to X$ be given. 
\begin{itemize}
\item[(i)] The {Riemann--Liouville fractional integral} of order $\beta$ of $f$ is defined by
$$
J_t^{\beta} f(t) := \dfrac{1}{\Gamma(\beta)}\int_0^t (t - s)^{\beta-1}f(s) \, ds,
$$
for almost every $t \in [0, T]$ such that the right-hand side is well defined.

\item[(ii)] The {Riemann--Liouville fractional derivative} of order $\beta$ of a function $f$ is defined as
$$
D_t^{\beta} f(t) := \dfrac{d^{[\beta]}}{dt^{[\beta]}} \left(J_t^{[\beta]-\beta} f(t)\right),
$$
for almost every $t \in [0, T]$ for which the right-hand side is well defined. Here, $[\beta]$ denotes the smallest integer greater than or equal to $\beta$, and the symbol $(d^{[\beta]}/dt^{[\beta]})$ stands for the ${[\beta]}$-th weak derivative of $f$.

\item[(iii)] The {Caputo fractional derivative} of order $\beta$ of $f$ is defined by
$$
cD_t^{\beta} f(t) := D_t^\beta\left[f(t) - \sum_{j=0}^{[\beta]-1} \left(\dfrac{t^j}{j!}\right) f^{(j)}(0)\right],
$$
for almost every $t \in [0, T]$ such that the right-hand side is well defined. Here, $f^{(j)}(0)$ denotes the $j$-th classical derivative of $f$ evaluated at $t=0$.
\end{itemize}

\end{definition}

From this point on, we present and, when necessary, prove all the results that are used in Section \ref{maintheorem}, the main section of this work, where we establish inequalities \eqref{202505271019} and \eqref{202505271020}.

\begin{proposition} \label{202507221111}
Let $0 < \alpha <1$ and $X$ a Banach space.
\begin{itemize}
\item[(i)] If $p>1/(1-\alpha)$ and $f\in L^p(0,T;X)$, it holds that
$$J_{t}^{1-\alpha} f(t)\big|_{t=0}=0.$$
\item[(ii)] If $f \in L^1(0,T;X)$
\begin{equation*}
D^\alpha_t \left[J_{t}^{{\alpha}}f(t)\right]=f(t), \textrm{ for a.e. } t \in [0,T].
\end{equation*}
If additionally $J_t^{1-{\alpha}}f \in W^{1,1}(0,T;X)$, then
\begin{equation*}J^\alpha_t \big[D^\alpha_t f(t)\big] = f(t) - \frac{t^{\alpha-1}}{\Gamma(\alpha)} \Big( J^{1-\alpha}_s f(s) |_{s=0} \Big)\end{equation*}
\item[(iii)] If $f \in C([0,T];X)$
\begin{equation*}
cD^\alpha_t \left[J_{t}^{{\alpha}}f(t)\right]=f(t), \textrm{ for every } t \in [0,T].
\end{equation*}
\item[(iv)] For $f \in W^{1,1}(0,T;X)$ we have
\begin{equation*}
cD_t^\alpha f (t) = J^{1-\alpha}_t f'(t), \textrm{ for a.e. } t \in [0,T].
\end{equation*}
\end{itemize}
\end{proposition}

\begin{proof} The proof of item $(i)$ can be found in \cite[Theorem 4.6]{CarFe0}. Items $(ii)$ and $(iii)$ are classical identities whose proof can be found in \cite[Proposition 2.35]{Car1}. Finally, item $(iv)$ is proved in \cite[Proposition 2.34]{Car1}.
\end{proof}

\begin{proposition}[{\cite[Theorem 3.1]{CarFe1}}]\label{202507221103} Let $\alpha>0$, $1\leq p\leq\infty$ and $f\in L^p(0,T;X)$. Then $J_{t}^\alpha f(t)$ is Bochner integrable and belongs to $L^p(0,T;X)$. Furthermore, it holds that
\begin{equation*}\left[\int_{0}^{T}{\left\|J_{t}^\alpha f(t)\right\|^p_X}\,dt\right]^{1/p}\leq \left[\dfrac{T^\alpha}{\Gamma(\alpha+1)}\right] \|f\|_{L^p(0,T;X)}.\end{equation*}
In other words, $J_{t}^\alpha$ is a bounded linear operator from the space $L^p(0,T;X)$ into itself.
\end{proposition}

\begin{proposition}[{\cite[Theorem 7]{CarFe4}}] \label{202507221156}
Consider $1<p<\infty$, $1/p<\alpha<\infty$ and let $f\in L^p(0,T;  X)$. Then $J_{t}^\alpha f\in C\big([0,T];X\big)$ and there exists $K_{\alpha,p,T}>0$ such that
\begin{equation*}\sup_{t\in[0,T]}\|J^\alpha_{t}f(t)\|_X\leq K_{\alpha,p,T}\left[\int_{0}^{T}{\|f(t)\|_X^p}\,dt\right]^{1/p}.\end{equation*}
As a consequence, $J_{t}^\alpha$ defines a bounded linear operator from the space $L^p(0,T;X)$ into the space $C\big([0,T];X\big)$. 

\end{proposition}

\begin{corollary}\label{202507221039}Let $0<\alpha<1$ and $\varphi \in C^\infty([0,T]; X)$ such that $\varphi^{(k-1)}(0) = 0$ for all $k \in \mathbb{N}:=\{1,2,\ldots\}$. Then $\psi := J_t^\alpha \varphi$ belongs to $C^\infty([0,T]; X)$ and satisfies $\psi^{(k-1)}(0) = 0$ for all $k \in \mathbb{N}$.
\end{corollary}
\begin{proof} From the regularity of $\varphi$, the condition $\varphi^{(k-1)}(0) = 0$ for all $k \in \mathbb{N}$, and item $(iv)$ of Proposition \ref{202507221111}, we deduce that
\begin{multline*}\psi^{(k-1)}(t)=\dfrac{d^{(k-1)}}{dt^{(k-1)}}\Big\{J^\alpha_t \big[\varphi(t)-\varphi(0)\big]\Big\}=\dfrac{d^{(k-2)}}{dt^{(k-2)}}\Big\{J^\alpha_t \varphi^\prime(t)\Big\}
\\=\dfrac{d^{(k-2)}}{dt^{(k-2)}}\Big\{J^\alpha_t \big[\varphi^\prime(t)-\varphi^\prime(0)\big]\Big\}=\dfrac{d^{(k-3)}}{dt^{(k-3)}}\Big\{J^\alpha_t \varphi^{\prime\prime}(t)\Big\}=\ldots=J^\alpha_t \varphi^{(k-1)}(t),\end{multline*}
for all $k \in \mathbb{N}$ and almost every $t\in[0,T]$. Then, Proposition \ref{202507221156} implies that $\psi \in C^\infty([0,T]; X)$. Moreover, by applying item (i) of Proposition \ref{202507221111}, we conclude that $\psi^{(k-1)}(0) = 0$ for all $k \in \mathbb{N}$.
\end{proof}

\begin{proposition}[{\cite[Theorem 5]{CarFe2}}]\label{202507231029}Consider $1<p<\infty$ and $0<\alpha<1/p$. If $f\in L^p(0,T;X)$, then $J_{t}^\alpha f(t)$ is Bochner integrable in $[0,T]$ and belongs to $L^{p/(1-p\alpha)}(0,T;X)$. Moreover, there exists $K_{\alpha,p,T}>0$ such that
\begin{equation*}\left[\int_{0}^T{\left\|J^\alpha_{t}f(t)\right\|_X^{p/(1-p\alpha)}}\,dt\right]^{(1-p\alpha)/p}\leq K_{\alpha,p,T}\left[\int_{0}^T{\|f(t)\|_X^p}\,dt\right]^{1/p}.\end{equation*}
In other words, $J_{t}^\alpha$ defines a bounded linear operator from the space $L^p(0,T;X)$ into the space $L^{p/(1-p\alpha)}(0,T;X)$.
\end{proposition}

\begin{theorem}[{\cite[Theorem 28]{CarFe3}}]\label{202507231041} For $1<p<\infty$ and $1\leq q<\infty$, there exists $K_{p,q}>0$ such that
\begin{equation*}\left[\int_{0}^T{\|J_{t}^{1/p}f(t)\|_X^q}\,dt\right]^{1/q}\leq K_{p,q}\left[\int_{0}^T{\|f(t)\|_X^p}\,dt\right]^{1/p},\end{equation*}
for every $f\in L^p(0, T; X)$. That is, $J_t^{1/p}$ defines a bounded linear operator from the space $L^p(0,T;X)$ into the space $L^q(0,T;X)$.
\end{theorem}

\section{Auxiliary Results and the Fractional Energy Inequality}
\label{maintheorem}

The goal of this section is to demonstrate that the most delicate condition in Theorem \ref{202507181644}, namely $J^{1-\alpha}_t \left\Vert u \right\Vert_H^2 \in W^{1,1}(0,T;\mathbb{R})$, is indeed satisfied under the additional assumptions outlined in the introduction. To make the discussion more concrete, we begin by assuming that $\Omega \subset \mathbb{R}^n$ is an open set and that $T > 0$ is fixed. In what follows, we derive inequalities \eqref{202505271019} and \eqref{202505271020} under appropriate hypotheses, thus extending the applicability of Theorem \ref{202507181644}, as anticipated from the outset of this work.

To keep this work self-contained, we start by recalling the following particular case of \cite[Chapter~3, Lemma~1.2]{Te1}.

\begin{proposition}\label{202504021854}
If $f \in W^{1,2} \left( 0,T; L^2(\Omega)\right)$, then
\begin{equation*}
\frac{d}{dt} \left\Vert f(t) \right\Vert^2_{L^2(\Omega)} = 2 \big( f'(t),f(t) \big)_{L^2(\Omega)},
\end{equation*}
for almost every $t\in[0,T]$. Note that the derivatives above are understood in the weak sense.
\end{proposition}

With this in mind, we now turn to the main results of this work: Theorem \ref{teodesdrl}, which deals with the Riemann--Liouville fractional derivative of order $\alpha \in (0,1)$, and Corollary \ref{teodesdcap}, which addresses the Caputo fractional derivative of the same order.

Let us now present a technical lemma.

\begin{lemma}\label{202504021931}Let $0<\alpha<1$. If $f \in W^{1,2} \left(0,T; L^2(\Omega) \right)$, then
\begin{multline*}2 \int_\Omega\int_0^t\frac{(t-s)^{-\alpha}}{\Gamma(1-\alpha)} f^\prime(x,s) \big[ f(x,t) - f(x,s) \big] ds\,dx
\\=\dfrac{\alpha}{\Gamma(1-\alpha)}\int_\Omega\int_0^t (t-w)^{\alpha-1} \left( \int_0^w (t-s)^{-\alpha} f'(x,s) \,ds \right)^2 dw\,dx,\end{multline*}
for almost every $t\in[0,T]$.
 \end{lemma}

\begin{proof} By applying the Fubini-Tonelli theorem, we obtain
\begin{multline}\label{202504021935}
\int_\Omega\int_0^t\frac{(t-s)^{-\alpha}}{\Gamma(1-\alpha)} f^\prime(x,s) \big[ f(x,t) - f(x,s) \big] ds\,dx \\
  = \int_\Omega\int_0^t\frac{(t-s)^{-\alpha}}{\Gamma(1-\alpha)} f^\prime(x,s) \int_s^t f^\prime(x,w) \,dw\,ds\,dx\\
 =  \int_\Omega\int_0^t  f^\prime(x,w) \int_0^w \frac{(t-s)^{-\alpha}}{\Gamma(1-\alpha)}   f^\prime(x,s)  \,ds\,dw\,dx,
\end{multline}
for almost every $t \in [0,T]$.

Now, noting that for almost every $w\in(0,t)$ and almost every $x\in\Omega$ we have
\begin{equation*}
\dfrac{(t-w)^{\alpha}}{2} \frac{d}{dw} \left[\left( \int_0^w (t-s)^{-\alpha} f^\prime(x,s)ds \right)^2\right] = f^\prime(x,w) \left( \int_0^w (t-s)^{-\alpha} f^\prime(x,s)ds \right),
\end{equation*}
we deduce from \eqref{202504021935} the identity
\begin{multline}\label{202504022003}
\int_\Omega\int_0^t\frac{(t-s)^{-\alpha}}{\Gamma(1-\alpha)} f^\prime(x,s) \big[ f(x,t) - f(x,s) \big] ds\,dx \\
= \dfrac{1}{2\Gamma(1-\alpha)}\int_\Omega\int_0^t (t-w)^{\alpha} \frac{d}{dw} \left( \int_0^w (t-s)^{-\alpha} f^\prime(x,s)ds \right)^2 dw\,dx,
\end{multline}
for almost every $t \in [0,T]$. 

At this point, notice that an integration by parts yields, 
\begin{multline}\label{202504022004}
\int_\Omega \int_0^t (t-w)^\alpha \frac{d}{dw} \left( \int_0^w (t-s)^{-\alpha} f'(x,s) \, ds \right)^2 dw\,dx \\
= \int_\Omega \left[ (t-w)^\alpha \left( \int_0^w (t-s)^{-\alpha} f'(x,s) \, ds \right)^2 \right] \Bigg|_{w=0}^{w=t} dx \\
+ \alpha \int_\Omega \int_0^t (t-w)^{\alpha - 1} \left( \int_0^w (t-s)^{-\alpha} f'(x,s) \, ds \right)^2 dw\,dx,
\end{multline}
for almost every $t \in [0, T]$.

We now show that the boundary terms vanish. Since $(t - w)^\alpha|_{w = t} = 0$, we have
\begin{equation}\label{202504022005}
\int_\Omega \left| (t-w)^\alpha \left( \int_0^w (t-s)^{-\alpha} f'(x,s) \, ds \right)^2 \right|_{w=t} dx = 0.
\end{equation}

Secondly, using $0 \leq s \leq w$ implies $(t-s)^{-\alpha} \leq (t-w)^{-\alpha}$, and applying the classical Minkowski's inequality for integrals (see \cite[Theorem 202]{HaLiPo1} for details), we obtain
\begin{multline*}
\left(\int_\Omega \left| \int_0^w (t-s)^{-\alpha} f'(x,s) \, ds \right|^2 dx\right)^{1/2}
\leq \int_0^w (t-s)^{-\alpha} \|f'(s)\|_{L^2(\Omega)} \, ds \\
\leq (t-w)^{-\alpha} \int_0^w \|f'(s)\|_{L^2(\Omega)} \, ds,
\end{multline*}
which implies
\begin{multline}\label{202505122158}
\int_\Omega \left| (t-w)^\alpha \left( \int_0^w (t-s)^{-\alpha} f'(x,s) \, ds \right)^2 \right|_{w=0} dx \\
\leq \lim_{w\rightarrow0}\left[(t-w)^{-\alpha} \left(\int_0^w \|f'(s)\|_{L^2(\Omega)} \, ds\right)^2\right]= 0.
\end{multline}

Combining \eqref{202504022004}, \eqref{202504022005}, and \eqref{202505122158}, along with \eqref{202504022003}, we conclude the desired identity.

\end{proof}

We now establish a preliminary, weaker version of inequality \eqref{202505271020}, using the lemma above. This intermediate result is instrumental in proving the main results that follow. It is worth emphasizing that the proof of this proposition was inspired by \cite[Lemma 1]{Al1}.

\begin{proposition}\label{202504031004}
Let $0<\alpha<1$. If $f \in W^{1,2} \left(0,T; L^2(\Omega) \right)$, then
\begin{equation*}
cD_t^\alpha \left\Vert f(t) \right\Vert^2_{L^2(\Omega)} \leq 2 \left( cD_t^\alpha f(t), f(t) \right)_{L^2(\Omega)},
\end{equation*}
\end{proposition}
for almost every $t\in[0,T]$.
\begin{proof}
Firstly, since $f \in W^{1,2} \left(0,T; L^2(\Omega) \right)$, item $(iv)$ of Proposition \ref{202507221111} and Proposition \ref{202504021854} ensure that we can express the Caputo derivative as
\begin{equation}\label{202504021901}
cD_t^\alpha \left\Vert f(t) \right\Vert_{L^2(\Omega)} ^2 = J^{1-\alpha}_t \left[\frac{d}{dt} \left\Vert f(t) \right\Vert_{L^2(\Omega)} ^2\right] = 2 J^{1-\alpha}_t \big( f^\prime(t), f(t) \big)_{L^2(\Omega)},
\end{equation}
for almost every $t \in [0,T]$.

On the other hand, item $(iv)$ of Proposition \ref{202507221111} also ensures that
\begin{equation}\label{202504021902}
\left( cD_t^\alpha f(t), f(t) \right)_{L^2(\Omega)} = \left( J^{1-\alpha}_t f^\prime(t), f(t) \right)_{L^2(\Omega)},
\end{equation}
for almost every $t \in [0,T]$. 

Combining \eqref{202504021901} and \eqref{202504021902}, we obtain
\begin{multline*}
2 \left( cD_t^\alpha f(t),f(t) \right)_{L^2(\Omega)} - cD_t^\alpha \left\Vert f(t) \right\Vert^2_{L^2(\Omega)} \\
 = 2 \int_{\Omega} \int_0^t \frac{(t-s)^{-\alpha}}{\Gamma(1-\alpha)} f^\prime(x,s)\,ds\,f(x,t)\,dx - 2  \int_0^t \frac{(t-s)^{-\alpha}}{\Gamma(1-\alpha)} \int_{\Omega} f^\prime(x,s) f(x,s)\,dx\, ds,
\end{multline*}
and, applying the Fubini-Tonelli Theorem, we conclude that
\begin{multline*}
2 \left( cD_t^\alpha f(t),f(t) \right)_{L^2(\Omega)} - cD_t^\alpha \left\Vert f(t) \right\Vert^2_{L^2(\Omega)}  \\= 2 \int_\Omega  \int_0^t\frac{(t-s)^{-\alpha}}{\Gamma(1-\alpha)} f^\prime(x,s) \big[ f(x,t) - f(x,s) \big]\,ds\,dx,
\end{multline*}
for almost every $t \in [0,T]$. Therefore, it follows from Lemma \ref{202504021931} that
\begin{multline*}
2 \left( cD_t^\alpha f(t),f(t) \right)_{L^2(\Omega)} - cD_t^\alpha \left\Vert f(t) \right\Vert^2_{L^2(\Omega)}  \\=\dfrac{1}{\Gamma(1-\alpha)}\int_\Omega\int_0^t \alpha (t-w)^{\alpha-1} \left( \int_0^w (t-s)^{-\alpha} f'(x,s) \,ds \right)^2 dw\,dx\geq0,
\end{multline*}
for almost every $t \in [0,T]$.
\end{proof}

The previous proposition established the first fractional counterpart of Proposition~\ref{202504021854}. We now aim to refine the result presented in Proposition~\ref{202504031004}. To this end, we begin by proving a few auxiliary results.

\begin{proposition}\label{202504031146}
Let $0 < \alpha < 1$ and $f \in L^\infty(0,T;X)$ such that $D^\alpha_t f \in L^2(0,T;X)$. Then there exists a sequence $(\psi_j)_{j\in \mathbb{N}}\subset C^\infty([0,T];X)$ satisfying $\psi^{(k-1)}_j(0) = 0$ for all $j,k \in \mathbb{N}$, such that 
\begin{equation*}
\psi_j \to f \quad \text{and} \quad D^\alpha_t \psi_j \to D^\alpha_t f, \text{ in } L^2(0,T;X).
\end{equation*}
\end{proposition}
\begin{proof}
Since $D^\alpha_t f \in L^2(0,T;X)$, there exists $(\phi_j)_{j\in \mathbb{N}} \subset C^\infty_c((0,T);X)$ with 
\begin{equation*}
\phi_j \to D^\alpha_t f, \quad \text{in }L^2(0,T;X).
\end{equation*}

Define $\psi_j(t) := J^\alpha_t \phi_j(t)$. Thanks to Corollary \ref{202507221039}, we have $\psi_j \in C^\infty([0,T]; X)$ with $\psi^{(k-1)}_j(0) = 0$, for all $j,k \in \mathbb{N}$. Moreover, Proposition \ref{202507221103} allows us to deduce that
\begin{equation}\label{202504031057}
\psi_j=J^\alpha_t \phi_j \to J^\alpha_t \big(D^\alpha_t f\big), \textrm{ in } L^2(0,T;X).
\end{equation}
However, since $f \in L^\infty(0,T; X)$, items $(i)$ and $(ii)$ of Proposition \ref{202507221111} ensure that
\begin{equation} \label{202507221142}
J^\alpha_t \big(D^\alpha_t f(t)\big) = f(t) - \frac{t^{\alpha-1}}{\Gamma(\alpha)} \left[ J^{1-\alpha}_s f(s) \right] \Big|_{s=0} = f(t),
\end{equation}
for almost every $t \in [0,T]$. Therefore, from \eqref{202504031057} and \eqref{202507221142}, we deduce the convergence
$$
\psi_j(t) \to f(t), \quad \text{in } L^2(0,T; X).
$$
Moreover, observe that, from item $(ii)$ of Proposition \ref{202507221111}, we also obtain
\begin{equation*}
D^\alpha_t \psi_j(t) = D^\alpha_t J^\alpha_t \phi_j(t) = \phi_j(t) \to D^\alpha_t f(t),\textrm{ in }L^2(0,T;X).
\end{equation*}

This completes the proof of the proposition.
\end{proof}

An analogous version of the previous proposition, but involving the Caputo fractional derivative, is the following:

\begin{corollary}
If $0<\alpha<1$ and $f \in C([0,T];X)$ is such that $cD_t^\alpha f \in L^2(0,T;X)$, then there exists $(\psi_j)_{j\in \mathbb{N}} \subset C^\infty([0,T];X)$ such that $\psi_j(0)=f(0)$, $\psi^{(k)}_j(0)=0$ for all $j,k \in \mathbb{N}$, and
\begin{equation*}
\psi_j \to f \quad \text{and} \quad cD_t^\alpha \psi_j \to cD_t^\alpha f, \text{ in } L^2(0,T;X).
\end{equation*}
\end{corollary}

\begin{proof}
If we define $g(t)=f(t)-f(0)$, then $g\in L^\infty(0,T;X)$ and $D_t^\alpha g=cD_t^\alpha f \in L^2(0,T;X)$. Thus, Proposition \ref{202504031146} ensures the existence of a sequence $({\tilde\psi}_j)_{j\in \mathbb{N}}$ that belongs to $C^\infty([0,T];X)$ such that ${\tilde\psi}^{(k-1)}_j(0) = 0$ for all $j,k \in \mathbb{N}$, and  
\begin{equation*}
{\tilde\psi}_j \to g \quad \text{and} \quad D^\alpha_t {\tilde\psi}_j \to D^\alpha_t g \text{ in } L^2(0,T;X).
\end{equation*}

Define $\psi_j(t)={\tilde\psi}_j(t)+f(0)$. Then the result follows.
\end{proof}

At this point in our work, we need to extend the result of Alsaedi et al. in \cite[Lemma 1]{AlAhKi} so that it can be applied to vector-valued functions. To this end, we first recall the notions of Hölder continuous and fractionally continuously differentiable functions, as well as the relationship between them (see \cite{CarFe05} for more details).
\begin{definition}\label{202507230946} Given $0 < \beta < 1$, we define:
\begin{itemize}
\item[(i)] $H^\beta([0,T];X) = \big \{ f\in C([0,T];X): \text{ there exists } M>0 \text{ such that }$ \\\hspace*{\fill} $ \big\|f(t)-f(s)\big\|_X\leq M{|t-s|^\beta}, \ \forall \, t,s\in[0,T] \big\};$\vspace*{0.2cm}
\item[(ii)] $C^\beta([0,T];X) = \big \{ f\in C([0,T];X): cD_t^\beta f\in C([0,T];X) \big\}.$
\end{itemize}
\end{definition}

\begin{proposition}[{\cite[Corollary 23]{CarFe05}}]\label{202507231000}
  Assume that $0<\gamma< 1$. Then $C^{\gamma}([t_0,t_1],X)\subset H^{\gamma}([t_0,t_1],X)$.
\end{proposition}

\begin{theorem} \label{teogenalsaedi}
Let $\alpha \in (0,1)$ and let $H$ be a Hilbert space with inner product $(\cdot,\cdot)_H$. If one of the following conditions is satisfied:
\begin{itemize}
\item[(i)] $f_1 \in C([0,T];H)$, the Riemann--Liouville fractional derivative of $f_1$ of order $\alpha$ exists at $t \in (0,T]$, and $f_2 \in H^\beta([0,T];H)$, with $0 < \alpha < \beta < 1$;

\item[(ii)] $f_1 \in H^\beta([0,T];H)$ and $f_2 \in H^\delta([0,T];H)$ for some $0 < \beta < 1$, $0 < \delta < 1$, with $0 < \alpha < \beta + \delta$, and the Riemann--Liouville fractional derivatives of both $f_1$ and $f_2$ of order $\alpha$ exist at $t \in (0,T]$,
\end{itemize}
then the following identity holds at each $t \in (0,T]$:
\begin{multline*}
D^\alpha_t \big( f_1(t),f_2(t) \big)_H = \big( D^\alpha_t f_1(t) , f_2(t) \big)_H + \big( f_1(t), D^\alpha_t f_2(t) \big)_H \\
-\frac{\alpha}{\Gamma(1-\alpha)} \int_0^t \frac{\big( f_1(s)-f_1(t), f_2(s)-f_2(t) \big)_H}{(t-s)^{\alpha+1}} ds - \frac{\big( f_1(t),f_2(t) \big)_H}{\Gamma(1-\alpha)t^\alpha},
\end{multline*}
\end{theorem}

\begin{proof}
We focus on the proof of item $(ii)$, as the argument for $(i)$ follows in a similar way. By definition, it holds that
\begin{multline} \label{duv}
D^\alpha_t \big( f_1(t), f_2(t) \big)_H =  \frac{1}{\Gamma(1-\alpha)} \left\{\frac{d}{dt}  \left[\int_0^t (t-s)^{-\alpha} \big( f_1(s), f_2(s) \big)_H\, ds\right]\right\} \\
\qquad\quad= \frac{1}{\Gamma(1-\alpha)} \left\{\lim_{\varepsilon \to 0} \frac{1}{\varepsilon} \left[ \int_0^{t+\varepsilon} (t+\varepsilon-s)^{-\alpha} \big( f_1(s), f_2(s) \big)_H\, ds\right.\right.\\\left.\left. - \int_0^{t} (t-s)^{-\alpha} \big( f_1(s), f_2(s) \big)_H\, ds \right]\right\}.
\end{multline}

Now notice that for $t,s \in [0,T]$ we have
\begin{multline} \label{uv}
\big( f_1(s), f_2(s) \big)_H = \big( f_1(s) - f_1(t), f_2(s)-f_2(t) \big)_H \\+ \big( f_1(s),f_2(t) \big)_H + \big( f_1(t),f_2(s) \big)_H - \big( f_1(t),f_2(t) \big)_H.
\end{multline}

Inserting \eqref{uv} into identity \eqref{duv}, and assuming that the limit is taken from the right of $0$ (i.e., $\lim_{\varepsilon \to 0^+}$), we obtain an expression that can be rearranged as the sum of four pairs of terms. For clarity, we denote these terms by 
\begin{multline*}\frac{1}{\Gamma(1-\alpha)} \left\{ \lim_{t \to 0^+} \frac{1}{\varepsilon} \left[\int_0^{t+\varepsilon} (t+\varepsilon -s)^{-\alpha} \big(f_1(s)-f_1(t), f_2(s)-f_2(t) \big)_H ds \right.\right.\\ 
\left.\left. - \int_0^{t} (t -s)^{-\alpha} \big(f_1(s)-f_1(t), f_2(s)-f_2(t) \big)_H\, ds \right]\right\}=:\mathcal{I}_1(t),\end{multline*}
\begin{multline*}\frac{1}{\Gamma(1-\alpha)} \left\{\lim_{\varepsilon \to 0^+} \frac{1}{\varepsilon} \left[ \int_0^{t+\varepsilon} (t+\varepsilon-s)^{-\alpha} \big( f_1(s),f_2(t) \big)_H\, ds\right.\right.\\ \left.\left. - \int_0^{t} (t-s)^{-\alpha} \big( f_1(s),f_2(t) \big)_H ds\, \right]\right\}=:\mathcal{I}_2(t),\end{multline*}
\begin{multline*}\frac{1}{\Gamma(1-\alpha)} \left\{\lim_{\varepsilon \to 0^+} \frac{1}{\varepsilon} \left[ \int_0^{t+\varepsilon} (t+\varepsilon-s)^{-\alpha} \big( f_1(t),f_2(s) \big)_H\, ds\right.\right.\\ \left.\left. - \int_0^{t} (t-s)^{-\alpha} \big( f_1(t),f_2(s) \big)_H ds\, \right]\right\}=:\mathcal{I}_3(t),\end{multline*}
and finally
\begin{multline*}\frac{1}{\Gamma(1-\alpha)} \left\{\lim_{\varepsilon \to 0^+} \frac{1}{\varepsilon} \left[ \int_0^{t+\varepsilon} (t+\varepsilon-s)^{-\alpha}\big( f_1(t),f_2(t) \big)_H\, ds\right.\right.\\\left.\left. - \int_0^{t} (t-s)^{-\alpha}\big( f_1(t),f_2(t) \big)_H\, ds \right]\right\}=:\mathcal{I}_4(t),\end{multline*}
for every $t \in (0,T)$ such that the four limits above exist. In fact, we shall prove that each of these limits does indeed exist.

We begin with $\mathcal{I}_2(t)$ and $\mathcal{I}_3(t)$, for which it is straightforward to verify that
\begin{equation*}
\mathcal{I}_2(t)=\big( D^\alpha_t f_1(t), f_2(t) \big)_H\qquad\textrm{ and }\qquad \mathcal{I}_3(t)=\big(  f_1(t), D^\alpha_t f_2(t) \big)_H.
\end{equation*}

For $\mathcal{I}_4(t)$, since we have 
\begin{multline*}
\lim_{\varepsilon \to 0^+} \frac{1}{\varepsilon} \left[ \int_0^{t+\varepsilon} (t+\varepsilon-s)^{-\alpha} ds - \int_0^{t} (t-s)^{-\alpha} ds \right] \\
= \lim_{\varepsilon \to 0^+} \frac{1}{\varepsilon} \left[ \frac{(t+\varepsilon)^{1-\alpha}}{1-\alpha} - \frac{t^{1-\alpha}}{1-\alpha}\right]=t^{-\alpha},
\end{multline*}
we obtain
$$\mathcal{I}_4(t)=\frac{\big( f_1(t),f_2(t) \big)_H}{\Gamma(1-\alpha)t^{\alpha}}.$$

Finally, let us address the most complex term, which requires a more careful analysis. First, we rewrite $\Gamma(1-\alpha) \mathcal{I}_1(t) = \mathcal{H}_1(t) + \mathcal{H}_2(t)$, where
$$\mathcal{H}_1(t)=\lim_{\varepsilon \to 0^+}\left[\frac{1}{\varepsilon} \int_t^{t+\varepsilon} (t+\varepsilon-s)^{-\alpha} \big( f_1(s)-f_1(t), f_2(s)-f_2(t) \big)_H\, ds  \right]$$
and
$$\mathcal{H}_2(t)=\lim_{\varepsilon \to 0^+}  \int_0^{t} \left[\frac{(t+\varepsilon-s)^{-\alpha} - (t-s)^{-\alpha}}{\varepsilon} \right] \big( f_1(s)-f_1(t), f_2(s)-f_2(t) \big)_H\, ds$$

To deal with $\mathcal{H}_1(t)$, we apply the change of variables $w = t + \varepsilon - s$, which gives
\begin{equation*}
\mathcal{H}_1(t)=\lim_{\varepsilon \to 0} \frac{1}{\varepsilon} \int_0^\varepsilon w^{-\alpha} \big( f_1(t+\varepsilon-w) -f_1(t), f_2(t+\varepsilon-w)-f_2(t)\big)_H\, dw.
\end{equation*}

Then, the Cauchy–Schwarz inequality, along with the Hölder regularity, ensures that
\begin{multline*}
\mathcal{H}_1(t)\leq\lim_{\varepsilon \to 0} \int_0^\varepsilon \frac{w^{-\alpha}}{\varepsilon} (\varepsilon-w)^{\beta+\delta}\,dw=\lim_{\varepsilon \to 0} \int_0^1 \frac{(\varepsilon h)^{-\alpha}}{\varepsilon} (\varepsilon-\varepsilon h)^{\beta+\delta} \varepsilon\,dh 
\\= \lim_{\varepsilon \to 0} \varepsilon^{\beta+\delta-\alpha} B(1-\alpha, \beta + \delta+1) =0,
\end{multline*}
where, above, $B(\cdot, \cdot)$ stands for the Beta function.

On the other hand, to handle $\mathcal{H}_2(t)$, let us first denote by $g_\varepsilon : [0,t] \rightarrow \mathbb{R}$ the integrand function that defines it, given by
$$g_\varepsilon(s):= \left[\frac{(t+\varepsilon-s)^{-\alpha} - (t-s)^{-\alpha}}{\varepsilon}\right] \, \big( f_1(s)-f_1(t), f_2(s)-f_2(t) \big)_H.$$

Note that  
\begin{equation*}
\lim_{\varepsilon \rightarrow 0^+} g_\varepsilon(s) = -\alpha(t-s)^{-\alpha-1} \big( f_1(s)-f_1(t), f_2(s)-f_2(t) \big)_H, \quad \text{for a.e. } s \in [0,t].
\end{equation*}
Moreover, we have
\begin{equation*}
\left\vert g_\varepsilon(s) \right\vert \leq \left\vert \frac{(t+\varepsilon-s)^{-\alpha}-(t-s)^{-\alpha}}{\varepsilon} \right\vert \, (t-s)^{\beta+\delta}, \quad \text{for a.e. } s \in [0,t],
\end{equation*}
thanks to the Hölder continuity of $f_1$ and $f_2$.

Consider now the function $h:[0,\varepsilon] \rightarrow \mathbb{R}$ given by $h(w) = (t+w-s)^{-\alpha}$. By the Mean Value Theorem, there exists $\theta \in (0,\varepsilon)$ such that  
\begin{multline*}
\left\vert \frac{(t+\varepsilon-s)^{-\alpha}-(t-s)^{-\alpha}}{\varepsilon} \right\vert = \left\vert \frac{h(\varepsilon)-h(0)}{\varepsilon} \right\vert = \left\vert h'(\theta) \right\vert \\= \alpha(t+\theta - s )^{-\alpha-1} \leq \alpha (t-s)^{-\alpha-1}
\end{multline*}
Therefore,
\begin{equation*}
\left\vert g_\varepsilon (t) \right\vert \leq (t-s)^{\beta+\delta -\alpha-1},
\end{equation*}
for almost every $s \in [0,t].$

The aforementioned properties are enough for us to apply the Dominated Convergence Theorem in order to obtain
\begin{equation*}
\mathcal{H}_2(t)=\lim_{\varepsilon\rightarrow0^+}\int_0^t g_\varepsilon (s) ds=\int_0^t -\alpha(t-s)^{-\alpha-1} \big( f_1(s) -f_1(t), f_2(s)-f_2(t) \big)_H\, ds,
\end{equation*}
for every $t\in(0,T)$. 

By combining the results obtained for $\mathcal{I}_1(t), \mathcal{I}_2(t), \mathcal{I}_3(t)$, and $\mathcal{I}_4(t)$ in \eqref{duv}, and noting that the limit from the left of $0$ can be treated analogously, we conclude the proof.
\end{proof}

Before presenting our main result, we recall one final theorem from Gorenflo et al. \cite{GoLuYa1} and generalize it to vector-valued functions. In that work, for $0 < \gamma < 1$, the authors consider the space $W^{\gamma,2}(0,T;\mathbb{R})$, which denotes the standard fractional Sobolev space (see \cite{NePaVa1} for more details), consisting of functions $f \in L^2(0,T;\mathbb{R})$ such that
$$
\big[f\big]_{W^{\gamma,2}(0,T;\mathbb{R})} := \left(\int_0^T \int_0^T \dfrac{|f(t) - f(s)|^2}{|t - s|^{2\gamma + 1}}\, ds\, dt \right)^{1/2} < \infty.
$$
Endowed with the norm
\begin{equation*}
\|f\|_{W^{\gamma,2}(0,T;\mathbb{R})} := \|f\|_{L^2(0,T;\mathbb{R})} + \big[f\big]_{W^{\gamma,2}(0,T;\mathbb{R})},
\end{equation*}
this space becomes a Banach space. In that context, they prove Theorem 2.1, which is the result we intend to employ in our final theorem. It states that there exist constants $m, M > 0$ such that
\begin{equation}\label{202504092059}
m\|J_t^\gamma f\|_{W^{\gamma,2}(0,T)} \leq \|f\|_{L^2(0,T)} \leq M \|J_t^\gamma f\|_{W^{\gamma,2}(0,T)},
\end{equation}
for every $f \in L^2(0,T)$.

Now let us consider the fractional Sobolev-Bochner space $W^{\gamma,2}(0,T;X)$ of all functions in $L^2(0,T;X)$ such that
\begin{equation}\label{202507231108}
\big[f\big]_{W^{\gamma,2}(0,T;X)} := \left(\int_0^T \int_0^T \dfrac{\|f(t) - f(s)\|_{X}^2}{|t - s|^{2\gamma + 1}}\, ds\, dt \right)^{1/2} < \infty.
\end{equation}
Like before, endowed with the norm
\begin{equation}\label{202504092058}
\|f\|_{W^{\gamma,2}(0,T;X)} := \sqrt{\|f\|^2_{L^2(0,T;X)} + \big[f\big]^2_{W^{\gamma,2}(0,T;X)}},
\end{equation}
it becomes a Banach space. This allows us to prove the following result.

\begin{corollary}\label{202504092126} There exist $m,M>0$ such that
$$m \|J_t^\gamma f\|_{W^{\gamma,2}(0,T;L^2(\Omega))} \leq \|f\|_{L^2(0,T;L^2(\Omega))} \leq M \|J_t^\gamma f\|_{W^{\gamma,2}(0,T;L^2(\Omega))},$$
for every $f \in L^2(0,T;L^2(\Omega))$.
\end{corollary}

\begin{proof} Let $f \in L^2(0,T;L^2(\Omega))$. Observe that, for almost every $x \in \Omega$, we have $f(x,\cdot) \in L^2(0,T)$. Therefore, from \eqref{202504092059}, it follows that
$$
m \|J_t^\gamma f(x,\cdot)\|_{W^{\gamma,2}(0,T)} \leq \|f(x,\cdot)\|_{L^2(0,T)} \leq M \|J_t^\gamma f(x,\cdot)\|_{W^{\gamma,2}(0,T)},
$$
for almost every $x \in \Omega$, with constants $m$ and $M$ independent of $x$. Squaring both sides of the inequality, then integrating it over $\Omega$ and applying the Fubini–Tonelli Theorem, we obtain the desired result.
\end{proof}

We have now established all the necessary results to prove our main theorems. It is worth emphasizing that the next theorem, which we consider the most challenging in the paper, is built from a sequence of ideas that are carefully detailed in the proof, in order to make the argument as clear and organized as possible.

\begin{theorem} \label{teodesdrl}
Assume that $0 < \alpha < 1$ and that $f \in L^\infty(0,T; L^2(\Omega))$ satisfies $D^\alpha_t f \in L^2(0,T; L^2(\Omega))$. Then $D^\alpha_t \|f\|_{L^2(\Omega)}^2 \in L^1(0,T)$ and
\begin{equation} \label{desa}
D^\alpha_t \|f\|_{L^2(\Omega)}^2 \leq 2 \big(D^\alpha_t f(t), f(t)\big)_{L^2(\Omega)},
\end{equation}
for almost every $t \in (0,T)$.
\end{theorem}

\begin{proof} To clarify the structure of the proof, we divide it into four steps. In the first step, we construct an auxiliary sequence $(\xi_j)_{j \in \mathbb{N}} \subset C^\infty([0,T];L^2(\Omega))$ such that 
$$
J^{1-\alpha}_t \left\| J^\alpha_t \xi^\prime_j \right\|_{L^2(\Omega)}^2 \to J^{1-\alpha}_t \left\| f \right\|_{L^2(\Omega)}^2, \quad \text{in } L^1(0,T).
$$
In the second step, we explain why it is necessary to prove that the sequence
$$
\left(D^\alpha_t \left\| J^\alpha_t \xi^\prime_j \right\|_{L^2(\Omega)}^2\right)_{j\in\mathbb{N}}
$$
is Cauchy in $L^1(0,T;\mathbb{R})$. In the third step, we show that the existence of such a limit allows us to conclude that $D^\alpha_t \| f \|_{L^2(\Omega)}^2$ exists and belongs to $\in L^1(0,T;\mathbb{R})$. Finally, in the fourth step, we prove inequality \eqref{desa}.
\vspace*{0.3cm}

\noindent\textbf{Step 1: Construction of an approximating sequence.}

We begin by recalling that Proposition \ref{202504031146} ensures the existence of a sequence $(\psi_j)_{j\in \mathbb{N}} \subset C^\infty([0,T];L^2(\Omega))$ such that $\psi^{(k-1)}_j(0) = 0$ for all $j,k \in \mathbb{N}$, and
\begin{equation}\label{202504071759}
\psi_j \to f \quad \text{and} \quad D^\alpha_t \psi_j \to D^\alpha_t f, \quad \text{in } L^2(0,T;L^2(\Omega)).
\end{equation}

Define the auxiliary sequence $\xi_j(t) := J^{1-\alpha}_t \psi_j(t)$ for all $j \in \mathbb{N}$. Observe that Corollary \ref{202507221039} ensures that  $(\xi_j)_{j \in \mathbb{N}} \subset C^\infty([0,T];L^2(\Omega))$ and satisfies $\xi_j^{(k-1)}(0) = 0$ for all $j, k \in \mathbb{N}$. Moreover, Proposition \ref{202507221103} guarantees the convergence
\begin{equation*}
\xi_j \to J^{1-\alpha}_t f, \quad \text{in } L^2(0,T;L^2(\Omega)).
\end{equation*}

Consequently, using the second convergence in \eqref{202504071759}, we obtain
\begin{equation} \label{psiltodf}
\xi^\prime_j = \frac{d}{dt}\Big[J^{1-\alpha}_t \psi_j\Big] = D^\alpha_t \psi_j \to D^\alpha_t f, \quad \text{in } L^2(0,T;L^2(\Omega)).
\end{equation}

Since $f \in L^\infty(0,T;L^2(\Omega))$, items $(i)$ and $(ii)$ of Proposition \ref{202507221111} ensure that $J^\alpha_t D^\alpha_t f(t) = f(t)$ for almost every $t \in [0,T]$. This identity, together with Proposition \ref{202507221103}, applied to \eqref{psiltodf}, allows us to conclude that
\begin{equation} \label{202504071804}
J^\alpha_t \xi^\prime_j \to J^\alpha_t D^\alpha_t f = f, \quad \text{in } L^2(0,T;L^2(\Omega)).
\end{equation}

Therefore, from \eqref{202504071804} and, once again using Proposition \ref{202507221103}, we obtain
\begin{equation} \label{convl1}
J^{1-\alpha}_t \left\| J^\alpha_t \xi^\prime_j \right\|_{L^2(\Omega)}^2 \to J^{1-\alpha}_t \left\| f \right\|_{L^2(\Omega)}^2, \quad \text{in } L^1(0,T;\mathbb{R}).\vspace*{0.1cm}
\end{equation}

\noindent\textbf{Step 2: The Cauchy sequence in $L^1(0,T)$.}

In this step, we aim to prove the existence of the weak derivative of $J_t^{1-\alpha} \|f\|_{L^2(\Omega)}^2$ and to show that it belongs to $L^1(0,T;\mathbb{R})$. This is the technical obstacle mentioned in the introduction.

To this end, we begin by observing that, by a recursive argument already used in the previous step, it holds that
\begin{equation}\label{202505092152}
J^\alpha_t \xi_j^\prime(t) = J^\alpha_t D^\alpha_t \psi_j(t) = \psi_j(t),
\end{equation}
for almost every $t \in [0,T]$. Combining this with the fact that $\psi_j(0) = 0$ for all $j \in \mathbb{N}$, and applying item $(iv)$ of Proposition \ref{202507221111} together with Proposition~\ref{202504021854}, we deduce
\begin{equation*}
D^\alpha_t \left\| J^\alpha_t \xi^\prime_j(t) \right\|_{L^2(\Omega)}^2 = \frac{d}{dt} \left[J^{1-\alpha}_t \left\| \psi_j(t) \right\|_{L^2(\Omega)}^2 \right]=2J^{1-\alpha}_t \left( \psi_j^\prime(t) , \psi_j(t) \right)_{L^2(\Omega)},
\end{equation*}
for almost every $t \in [0,T]$. Hence, 
$$D^\alpha_t \left\| J^\alpha_t \xi^\prime_j(\cdot) \right\|_{L^2(\Omega)}^2\in L^1(0,T).$$

Therefore, for any $\theta \in C_c^{\infty}((0,T);\mathbb{R})$, it follows that
\begin{equation}\label{202504081139}
\int_0^T J^{1-\alpha}_t \left\| J^\alpha_t \xi^\prime_j(t) \right\|_{L^2(\Omega)}^2 \theta'(t) dt = -\int_0^T \left[D^{\alpha}_t \left\| J^\alpha_t \xi^\prime_j(t) \right\|_{L^2(\Omega)}^2\right] \theta(t) dt.
\end{equation}
and from \eqref{convl1}
\begin{equation}\label{202504081140}
\int_0^T J^{1-\alpha}_t \left\| J^\alpha_t \xi^\prime_j(t) \right\|_{L^2(\Omega)}^2 \theta^\prime(t) dt \to \int_0^T J^{1-\alpha}_t \left\| f(t) \right\|_{L^2(\Omega)}^2 \theta^\prime(t) dt.
\end{equation}

Now, if we show that
\begin{equation}\label{202504081141}
\left( D^\alpha_t \left\| J^\alpha_t \xi_j' \right\|_{L^2(\Omega)}^2 \right)_{j \in \mathbb{N}}
\end{equation}
is a Cauchy sequence in the space $L^1(0,T; \mathbb{R})$, then there exists a function $g \in L^1(0,T; \mathbb{R})$ such that
\begin{equation}\label{202504081142}
D^\alpha_t \left\| J^\alpha_t \xi_j' \right\|_{L^2(\Omega)}^2 \to g, \quad \text{in } L^1(0,T; \mathbb{R}).
\end{equation}
Consequently, combining \eqref{202504081139}, \eqref{202504081140}, and \eqref{202504081142}, we obtain
$$
\int_0^T g(t) \theta(t) \, dt = - \int_0^T J^{1-\alpha}_t \left| f(t) \right|^2 \theta'(t) \, dt,
$$
for every $\theta \in C_c^\infty((0,T); \mathbb{R})$. In other words, the weak fractional derivative $D^\alpha_t \left\| f(t) \right\|_{L^2(\Omega)}^2$ exists and coincides with $g(t)$ for almost every $t \in [0,T]$.\vspace*{0.3cm}

\noindent\textbf{Step 3: Existence of the fractional derivative.}

The previous step shows that, to complete the first part of the proof, it remains to verify that \eqref{202504081141} forms a Cauchy sequence in $L^1(0,T; \mathbb{R})$. We now turn our attention to this task. To begin, observe that for every $j, k \in \mathbb{N}$,
\begin{equation}\label{202404081539}
\left\| J^\alpha_t \xi'_j(t) \right\|_{L^2(\Omega)}^2 - \left\| J^\alpha_t \xi'_k(t)  \right\|_{L^2(\Omega)}^2 
 =  \big( J^\alpha_t \xi'_j(t) - J^\alpha_t \xi'_k(t)  , J^\alpha_t \xi'_j(t) + J^\alpha_t \xi'_k(t)  \big)_{L^2(\Omega)},
\end{equation}
for almost every $t \in [0,T]$. To simplify the notation, we temporarily define
\begin{equation}\label{202507231111}
u_{j,k}(t) := J^\alpha_t \xi'_j(t) - J^\alpha_t \xi'_k(t) \quad \text{and} \quad v_{j,k}(t) := J^\alpha_t \xi'_j(t) + J^\alpha_t \xi'_k(t).
\end{equation}

From item $(iii)$ of Proposition \ref{202507221111}, we know that the functions $cD_t^\alpha u_{j,k}$ and $cD_t^\alpha v_{j,k}$ belong to $C([0,T]; L^2(\Omega))$, which implies, by item $(ii)$ of Definition \ref{202507230946}, that $u_{j,k}, v_{j,k} \in C^\alpha([0,T]; L^2(\Omega))$. Then, it follows from Proposition \ref{202507231000} that both functions, in fact, belong to $H^\gamma([0,T]; L^2(\Omega))$ for every $0 < \gamma < \alpha$. 

Therefore, applying item $(ii)$ of Theorem \ref{teogenalsaedi} to equation \eqref{202404081539}, we obtain

\begin{multline*}
D^\alpha_t \big\| J^\alpha_t \xi^\prime_j(t) \big\|_{L^2(\Omega)}^2 - D^\alpha_t \big\| J^\alpha_t \xi^\prime_k(t)  \big\|_{L^2(\Omega)}^2 = D^\alpha_t \big( u_{j,k}(t),v_{j,k}(t)  \big)_{L^2(\Omega)} \\
= \big( D^\alpha_t u_{j,k}(t), v_{j,k}(t) \big)_{L^2(\Omega)} + \big( u_{j,k}(t) , D^\alpha_t v_{j,k}(t) \big)_{L^2(\Omega)} - \frac{\big( u_{j,k}(t),v_{j,k}(t) \big)_{L^2(\Omega)}}{\Gamma(1-\alpha)t^{\alpha}} \\
- \frac{\alpha}{\Gamma(1-\alpha)} \int_0^t  \frac{\big( u_{j,k}(s) -u_{j,k}(t), v_{j,k}(s)-v_{j,k}(t)  \big)_{L^2(\Omega)}}{(t-s)^{\alpha+1}} ds.
\end{multline*}
for every $t \in (0,T]$. In other words,
\begin{equation*} 
D^\alpha_t \big\| J^\alpha_t \xi^\prime_j(t) \big\|_{L^2(\Omega)}^2 - D^\alpha_t \big\| J^\alpha_t \xi^\prime_k(t)  \big\|_{L^2(\Omega)}^2
= \mathcal{I}_{jk}(t)+\mathcal{J}_{jk}(t)+\mathcal{K}_{jk}(t)+\mathcal{L}_{jk}(t),
\end{equation*}
for every $t \in (0,T]$, where
\begin{equation*}\mathcal{I}_{jk}(t):=\int_\Omega \left[ \xi'_j(x,t) - \xi'_k(x,t) \right] \cdot J^\alpha_t \left[  \xi'_j(x,t) +  \xi'_k(x,t)  \right]\,dx ,\end{equation*}
\begin{equation*}\mathcal{J}_{jk}(t):= \int_\Omega J^\alpha_t\left[ \xi'_j(x,t) - \xi'_k(x,t) \right] \cdot \left[  \xi'_j(x,t) +  \xi'_k(x,t)  \right]\,dx,\end{equation*}
\begin{equation*}\mathcal{K}_{jk}(t):= - \frac{1}{\Gamma(1-\alpha) t^\alpha} \int_\Omega J^\alpha_t \left[ \xi'_j(x,t) - \xi'_k(x,t) \right] \cdot J^\alpha_t \left[ \xi'_j(x,t) + \xi'_k(x,t) \right]\,dx,\end{equation*}
and
\begin{multline*}- \frac{\alpha}{\Gamma(1-\alpha)} \int_0^t\int_\Omega  \left\{\frac{J^\alpha_s \left[ \xi'_j(x,s) - \xi'_k(x,s) \right]-J^\alpha_t \left[ \xi'_j(x,t) - \xi'_k(x,t) \right]}{(t-s)^{\frac{\alpha+1}{2}}} \right\} \\
\times \left\{\frac{J^\alpha_s \left[ \xi'_j(x,s) + \xi'_k(x,s) \right]-J^\alpha_t \left[ \xi'_j(x,t) + \xi'_k(x,t) \right]}{(t-s)^{\frac{\alpha+1}{2}}} \right\} dx\,ds=:\mathcal{L}_{jk}(t).\end{multline*}

Let us now analyze each of these terms. First, observe that Hölder's inequality yields
$$
|\mathcal{I}_{jk}(t)| \leq \big\|\xi_j^\prime(t) - \xi_k^\prime(t)\big\|_{L^2(\Omega)} \, \big\|J_t^\alpha[\xi_j^\prime(t) - \xi_k^\prime(t)]\big\|_{L^2(\Omega)},
$$
for every $t \in (0,T]$. Then, we integrate with respect to the variable $t$ over $[0,T]$ and apply Hölder's inequality once again, thereby obtaining
$$
\|\mathcal{I}_{jk}\|_{L^1(0,T;\mathbb{R})} \leq \|\xi_j^\prime - \xi_k^\prime\|_{L^2(0,T;L^2(\Omega))} \cdot \|J_t^\alpha[\xi_j^\prime - \xi_k^\prime]\|_{L^2(0,T;L^2(\Omega))}.
$$
Hence, it follows from \eqref{psiltodf} and \eqref{202504071804} that the values $\big(\|\mathcal{I}_{jk}\|_{L^1(0,T;\mathbb{R})}\big)_{j,k}$ tends to zero as $j$ and $k$ increase. A similar argument shows that the same holds for $\big(\|\mathcal{J}_{jk}\|_{L^1(0,T;\mathbb{R})}\big)_{j,k}$.

Regarding the third term, we apply Cauchy--Schwarz inequality, to obtain
\begin{multline}\label{202504081729}
|\mathcal{K}_{jk}(t)| \leq \int_\Omega \left\{ \frac{J^\alpha_t \left| \xi^\prime_j(x,t) - \xi^\prime_k(x,t) \right| }{[\Gamma(1-\alpha) t^\alpha]^\frac{1}{2}} \right\}\left\{ \frac{ J^\alpha_t \left| \xi^\prime_j(x,t) + \xi^\prime_k(x,t) \right| }{[\Gamma(1-\alpha) t^\alpha]^\frac{1}{2}} \right\} \,dx \\\leq\left\{\dfrac{\big\|J_t^\alpha\big[\xi^\prime_j(t) - \xi^\prime_k(t)\big]\big\|_{L^2(\Omega)}}{[\Gamma(1-\alpha) t^\alpha]^\frac{1}{2}}\right\}
\left\{\dfrac{\big\|J_t^\alpha\big[\xi^\prime_j(t) + \xi^\prime_k(t)\big]\big\|_{L^2(\Omega)}}{[\Gamma(1-\alpha) t^\alpha]^\frac{1}{2}}\right\},
\end{multline}
for every $t \in (0,T]$. 
\begin{itemize}
\item[(i)] Assume for a moment that $0 < \alpha \leq 1/2$. In this case, by integrating with respect to the variable $t$ over $[0,T]$ and applying Hölder's inequality, we obtain from \eqref{202504081729} that
\begin{multline}\label{202504091844}
\qquad\quad\|\mathcal{K}_{jk}\|_{L^1(0,T;\mathbb{R})} \leq \left( \int_0^T \dfrac{t^{-\alpha} \left\| J_t^\alpha\left[\xi^\prime_j(t) - \xi^\prime_k(t)\right] \right\|^2_{L^2(\Omega)}}{\Gamma(1-\alpha)}\,dt \right)^{\frac{1}{2}} \\
\times \left( \int_0^T \dfrac{t^{-\alpha} \left\| J_t^\alpha\left[\xi^\prime_j(t) + \xi^\prime_k(t)\right] \right\|^2_{L^2(\Omega)}}{\Gamma(1-\alpha)}\,dt \right)^{\frac{1}{2}}.
\end{multline}
\begin{itemize}
\item[(a)] If $0<\alpha<1/2$, by applying H\"{o}lder's inequality with exponents $p = 1/(2\alpha)$ and $p' = 1/(1 - 2\alpha)$ to each integral on the right-hand side of \eqref{202504091844}, we obtain
\begin{multline*}
\qquad\qquad\quad\|\mathcal{K}_{jk}\|_{L^1(0,T;\mathbb{R})} \\\leq \dfrac{1}{\Gamma(1-\alpha)} \left( \int_0^T t^{-\frac{1}{2}}\,dt \right)^{2\alpha} \left\| J_t^\alpha\left[\xi^\prime_j - \xi^\prime_k\right] \right\|_{L^{\frac{2}{1 - 2\alpha}}(0,T;L^2(\Omega))} \\
\times \left\| J_t^\alpha\left[\xi^\prime_j + \xi^\prime_k\right] \right\|_{L^{\frac{2}{1 - 2\alpha}}(0,T;L^2(\Omega))}.
\end{multline*}
Since Proposition \ref{202507231029} ensures that the operator 
$$
J^\alpha_t : L^2(0,T;L^2(\Omega)) \to L^{\frac{2}{1 - 2\alpha}}(0,T;L^2(\Omega)),
$$
is a linear bounded operator, we finally obtain
\begin{equation}\label{202504091840}
\|\mathcal{K}_{jk}\|_{L^1(0,T;\mathbb{R})} \leq C_\alpha \left\| \xi^\prime_j - \xi^\prime_k \right\|_{L^2(0,T;L^2(\Omega))} \left\| \xi^\prime_j + \xi^\prime_k \right\|_{L^2(0,T;L^2(\Omega))}.
\end{equation}
Therefore, it follows from \eqref{psiltodf} that $\big(\|\mathcal{K}_{jk}\|_{L^1(0,T;\mathbb{R})}\big)_{j,k}$ tends to zero as $j$ and $k$ increase.\vspace*{0.2cm}

\item[(b)] If $\alpha=1/2$, by applying H\"{o}lder's inequality with exponents $p = 3/2$ and $p' = 3$ to each integral on the right-hand side of \eqref{202504091844}, we obtain
\begin{multline*}
\qquad\qquad\quad\|\mathcal{K}_{jk}\|_{L^1(0,T;\mathbb{R})} \\\leq \dfrac{1}{\Gamma(1/2)} \left( \int_0^T t^{-\frac{3}{4}}\,dt \right)^{\frac{2}{3}} \left\| J_t^\alpha\left[\xi^\prime_j - \xi^\prime_k\right] \right\|_{L^{6}(0,T;L^2(\Omega))} \\
\times \left\| J_t^\alpha\left[\xi^\prime_j + \xi^\prime_k\right] \right\|_{L^{6}(0,T;L^2(\Omega))}.
\end{multline*}
Since Proposition \ref{202507231041} ensures that the operator 
$$
J^\frac{1}{2}_t : L^2(0,T;L^2(\Omega)) \to L^{6}(0,T;L^2(\Omega)),
$$
is a linear bounded operator, we finally obtain
\begin{equation*}
\|\mathcal{K}_{jk}\|_{L^1(0,T;\mathbb{R})} \leq C \left\| \xi^\prime_j - \xi^\prime_k \right\|_{L^2(0,T;L^2(\Omega))} \left\| \xi^\prime_j + \xi^\prime_k \right\|_{L^2(0,T;L^2(\Omega))}.
\end{equation*}
Therefore, it follows from \eqref{psiltodf} that $\big(\|\mathcal{K}_{jk}\|_{L^1(0,T;\mathbb{R})}\big)_{j,k}$ tends to zero as $j$ and $k$ increase.
\vspace*{0.2cm}

\item[(ii)] When $1/2<\alpha<1$, recall that Proposition \ref{202507221156} ensures that
$$
J^\alpha_t : L^2(0,T;L^2(\Omega)) \to C([0,T];L^2(\Omega)),
$$
is a bounded and linear operator. Therefore, it is not difficult to proceed analogously to the previous items and conclude that the sequence $\big(\|\mathcal{K}_{jk}\|_{L^1(0,T;\mathbb{R})}\big)_{j,k}$ tends to zero as $j$ and $k$ increase.\vspace*{0.2cm}

\end{itemize}
\end{itemize}

To conclude that \eqref{202504081141} defines a Cauchy sequence, it remains to estimate the $L^1(0,T; \mathbb{R})$ norm of the fourth term, $\mathcal{L}_{jk}$. For every $t, s \in (0,T)$ with $0 < s < t$, by H\"older's inequality we have
\begin{multline*}
\left|\int_\Omega \left\{\frac{J^\alpha_s \left[ \xi'_j(x,s) - \xi'_k(x,s) \right] - J^\alpha_t \left[ \xi'_j(x,t) - \xi'_k(x,t) \right]}{(t-s)^{\frac{\alpha+1}{2}}} \right\} \right. \\
\left. \times \left\{\frac{J^\alpha_s \left[ \xi'_j(x,s) + \xi'_k(x,s) \right] - J^\alpha_t \left[ \xi'_j(x,t) + \xi'_k(x,t) \right]}{(t-s)^{\frac{\alpha+1}{2}}} \right\} dx \right|
\end{multline*}
is bounded above by
\begin{multline*}
\dfrac{\left\|J^\alpha_s \left( \xi'_j(s) - \xi'_k(s) \right) - J^\alpha_t \left( \xi'_j(t) - \xi'_k(t) \right)\right\|_{L^2(\Omega)}}{(t-s)^{\frac{\alpha+1}{2}}} \\
\times \dfrac{\left\|J^\alpha_s \left( \xi'_j(s) + \xi'_k(s) \right) - J^\alpha_t \left( \xi'_j(t) + \xi'_k(t) \right)\right\|_{L^2(\Omega)}}{(t-s)^{\frac{\alpha+1}{2}}},
\end{multline*}
so that, by integrating with respect to the variable $s$ over $[0,t]$, then with respect to $t$ over $[0,T]$, and applying Hölder's inequality twice in succession, we obtain the next inequality (the full expression can be found by recalling identity \ref{202507231111})
\begin{multline*}
\|\mathcal{L}_{jk}\|_{L^1(0,T;\mathbb{R})} \\\leq \dfrac{\alpha}{\Gamma(1-\alpha)} \left\{ \int_0^T \int_0^t (t-s)^{-\alpha-1} \left\| u_{j,k}(s) - u_{j,k}(t) \right\|^2_{L^2(\Omega)}\, ds\, dt \right\}^{1/2} \\
\times \left\{ \int_0^T \int_0^t (t-s)^{-\alpha-1} \left\| v_{j,k}(s) - v_{j,k}(t) \right\|^2_{L^2(\Omega)}\, ds\, dt \right\}^{1/2}.
\end{multline*}
Combined with the definition of the norm in $W^{\frac{\alpha}{2},2}(0,T; L^2(\Omega))$ (see identities \eqref{202507231108} and \eqref{202504092058} for details), this yields
\begin{multline*}
\|\mathcal{L}_{jk}\|_{L^1(0,T;\mathbb{R})}\\ \leq \dfrac{\alpha}{\Gamma(1-\alpha)} \left\| J^\alpha_t \big[\xi'_j - \xi'_k\big] \right\|_{W^{\frac{\alpha}{2},2}(0,T;L^2(\Omega))} \left\| J^\alpha_t \big[\xi'_j + \xi'_k\big] \right\|_{W^{\frac{\alpha}{2},2}(0,T;L^2(\Omega))}.
\end{multline*}
By recalling the semigroup property of the Riemann-Liouville fractional integral and applying Corollary~\ref{202504092126} we obtain
\begin{multline*}
\|\mathcal{L}_{jk}\|_{L^1(0,T;\mathbb{R})} \\\leq \dfrac{\alpha}{m^2 \Gamma(1-\alpha)} \left\| J^{\frac{\alpha}{2}}_t \big[\xi'_j - \xi'_k\big] \right\|_{L^2(0,T;L^2(\Omega))} \left\| J^{\frac{\alpha}{2}}_t \big[\xi'_j + \xi'_k\big] \right\|_{L^2(0,T;L^2(\Omega))}.
\end{multline*}

Finally, Proposition \ref{202507221103} ensures that
$$
\|\mathcal{L}_{jk}\|_{L^1(0,T;\mathbb{R})} \leq N_\alpha \left\| \xi'_j - \xi'_k \right\|_{L^2(0,T;L^2(\Omega))} \left\| \xi'_j + \xi'_k \right\|_{L^2(0,T;L^2(\Omega))}.
$$
As before, we conclude that the sequence $\big(\|\mathcal{L}_{jk}\|_{L^1(0,T;\mathbb{R})}\big)_{j,k}$ tends to zero as $j$ and $k$ increase.

Taking into account the entire discussion above, we conclude that \eqref{202504081141} defines a Cauchy sequence in $L^1(0,T;\mathbb{R})$. Consequently, by Step 2, we obtain
\begin{equation*}
D^\alpha_t \left\| f \right\|_{L^2(\Omega)}^2 = g \in L^1(0,T;\mathbb{R}).\vspace*{0.2cm}
\end{equation*}

\noindent\textbf{Step 4: The proof of inequality \eqref{desa}.}

We begin this final step by recalling from \eqref{202504071759}, \eqref{202505092152}, and \eqref{202504081142} the existence of a sequence $(\psi_j)_{j\in \mathbb{N}} \subset C^\infty([0,T];L^2(\Omega))$ satisfying $\psi^{(k-1)}_j(0) = 0$ for every $j,k \in \mathbb{N}$, and such that
\begin{equation}\label{202505092250}
\psi_j \to f \quad \text{and} \quad D^\alpha_t \psi_j \to D^\alpha_t f, \quad \text{in } L^2(0,T;L^2(\Omega)),
\end{equation}
as well as
\begin{equation}\label{202505092255}
D^\alpha_t \|\psi_j\|^2_{L^2(\Omega)} \to D^\alpha_t \|f\|_{L^2(\Omega)}^2, \quad \text{in } L^1(0,T;\mathbb{R}).
\end{equation}

Now observe that
\begin{multline*}
\left| \big( D^\alpha_t \psi_j(t), \psi_j(t) \big)_{L^2(\Omega)} - \big( D^\alpha_t f(t), f(t) \big)_{L^2(\Omega)} \right| \\
= \left| \big( D^\alpha_t \psi_j(t) - D^\alpha_t f(t), \psi_j(t) \big)_{L^2(\Omega)} + \big( D^\alpha_t f(t), \psi_j(t) - f(t) \big)_{L^2(\Omega)} \right|,
\end{multline*}
for almost every $t \in (0,T)$. Applying the Cauchy--Schwarz inequality, we obtain
\begin{multline}\label{202505092251}
\left| \big( D^\alpha_t \psi_j(t), \psi_j(t) \big)_{L^2(\Omega)} - \big( D^\alpha_t f(t), f(t) \big)_{L^2(\Omega)} \right| \\
\leq \left\| D^\alpha_t \psi_j(t) - D^\alpha_t f(t) \right\|_{L^2(\Omega)} \left\| \psi_j(t) \right\|_{L^2(\Omega)} + \left\| D^\alpha_t f(t) \right\|_{L^2(\Omega)} \left\| \psi_j(t) - f(t) \right\|_{L^2(\Omega)},
\end{multline}
for almost every $t \in (0,T)$. Consequently, by combining \eqref{202505092250} and \eqref{202505092251}, it follows that
\begin{equation}\label{202505092256}
\big( D^\alpha_t \psi_j , \psi_j \big)_{L^2(\Omega)} \to \big( D^\alpha_t f , f \big)_{L^2(\Omega)} \quad \text{in } L^1(0,T;\mathbb{R}).
\end{equation}

As a consequence of \eqref{202505092255} and \eqref{202505092256}, we may extract a subsequence $(\psi_{j_k})_{k \in \mathbb{N}}$ from $(\psi_{j})_{j \in \mathbb{N}}$, such that
\begin{equation}\label{202505092258}
D^\alpha_t \|\psi_{j_k}\|^2_{L^2(\Omega)} \to D^\alpha_t \|f\|_{L^2(\Omega)}^2 \quad \text{and} \quad \big( D^\alpha_t \psi_{j_k} , \psi_{j_k} \big)_{L^2(\Omega)} \to \big( D^\alpha_t f , f \big)_{L^2(\Omega)},
\end{equation}
for almost every $t \in (0,T)$. Finally, recall that Proposition \ref{202504031004} guarantees that
\begin{equation*}
D^\alpha_t \left\| \psi_{j_k}(t) \right\|_{L^2(\Omega)}^2 \leq 2 \big( D^\alpha_t \psi_{j_k}(t), \psi_{j_k}(t) \big)_{L^2(\Omega)},
\end{equation*}
for all $k\in\mathbb{N}$ and almost every $t \in (0,T)$. Therefore, combining the pointwise limits in \eqref{202505092258} with the inequality above, we conclude the desired result.
\end{proof}

Before stating our final result, we find it useful to include a brief remark, as noted by Carvalho-Neto and Fehlberg Júnior in \cite[Remark 4.17]{CarFe0}, concerning the well-posedness of the Caputo fractional derivative $cD_t^\alpha f$ for functions $f \in L^2(0,T; \mathbb{R})$.

In general, the value $f(0)$ may not be well defined for functions in $L^2(0,T; X)$, which poses a technical issue in defining the Caputo fractional derivative at the initial time. To overcome this difficulty, some authors assume $f \in C([0,T]; X)$, but this requirement is stronger than necessary. In our setting, it suffices to assume that $f$ admits a continuous representative on an interval $[0,\delta]$, for some $0<\delta\leq T$.

As a consequence of Theorem \ref{teodesdrl}, we obtain the following result.

\begin{corollary} \label{teodesdcap}
Let $0 < \alpha < 1$ and suppose $f \in L^\infty(0,T; L^2(\Omega)) \cap C([0,\delta]; L^2(\Omega))$ for some $0 < \delta \leq T$, with $cD_t^\alpha f \in L^2(0,T; L^2(\Omega))$. Then $cD_t^\alpha \| f \|_{L^2(\Omega)}^2 \in L^1(0,T; \mathbb{R})$, and
\begin{equation*}
cD_t^\alpha \left\| f(t) \right\|_{L^2(\Omega)}^2 \leq 2 \big( cD_t^\alpha f(t), f(t) \big)_{L^2(\Omega)},
\end{equation*}
for almost every $t \in [0,T]$.
\end{corollary}

\begin{proof}
For this proof, let us define the auxiliary function $g(t) := f(t) - f(0)$. Observe that $g \in L^\infty(0,T;L^2(\Omega))$ and $D^\alpha_t g = cD_t^\alpha f \in L^2(0,T;L^2(\Omega))$. By Theorem~\ref{teodesdrl}, it follows that $D^\alpha_t \|g\|_{L^2(\Omega)}^2 \in L^1(0,T;\mathbb{R})$ and
\begin{equation}\label{teodesdcap0}
D^\alpha_t \|g(t)\|_{L^2(\Omega)}^2 \leq 2 \big( D^\alpha_t g(t), g(t) \big)_{L^2(\Omega)},
\end{equation}
for almost every $t \in [0,T]$.

Note that
\begin{equation} \label{teodesdcap1}
\big( D^\alpha_t g(t), g(t) \big)_{L^2(\Omega)} = \big( cD_t^\alpha f(t), f(t) \big)_{L^2(\Omega)} - \big( D^\alpha_t [f(t) - f(0)], f(0) \big)_{L^2(\Omega)},
\end{equation}
for almost every $t \in [0,T]$. On the other hand,
\begin{multline*}
D^\alpha_t \|g(t)\|_{L^2(\Omega)}^2 = D^\alpha_t \int_\Omega |f(x,t) - f(x,0)|^2 \, dx \\
= \int_\Omega D^\alpha_t [f(x,t)]^2 - 2 f(x,0) D^\alpha_t f(x,t) + D^\alpha_t [f(x,0)]^2 \, dx \\
= \int_\Omega D^\alpha_t \left( [f(x,t)]^2 - [f(x,0)]^2 \right) - 2 f(x,0) D^\alpha_t f(x,t) + 2 D^\alpha_t [f(x,0)]^2 \, dx,
\end{multline*}
for almost every $t \in [0,T]$, which leads to
\begin{equation}\label{teodesdcap2}
D^\alpha_t \|g(t)\|_{L^2(\Omega)}^2 = cD_t^\alpha \|f(t)\|_{L^2(\Omega)}^2 - 2 \big( D^\alpha_t [f(t) - f(0)], f(0) \big)_{L^2(\Omega)},
\end{equation}
for almost every $t \in [0,T]$.

Comparing \eqref{teodesdcap0}, \eqref{teodesdcap1} and \eqref{teodesdcap2}, we conclude the proof.
\end{proof}

\section{The Fractional Heat Equation}\label{fracheat}

In this final section, we apply the Faedo–Galerkin method together with Corollary \ref{teodesdcap} to the fractional heat equation in order to prove the existence of a weak solution. 

We emphasize that standard notations for Sobolev spaces, Lebesgue spaces, and test functions are adopted throughout this last section. As these are classical, we refer the reader to \cite[Chapter~5]{Ev01} for precise definitions and notations. Furthermore, routine computations are omitted to keep the exposition concise and focused on the essential new arguments required by the fractional setting.

Consider $1/2<\alpha<1$, $\Omega$ an open and bounded subset of $\mathbb{R}^n$, $T>0$ fixed, $f:[0,T]\rightarrow L^2(\Omega)$ such that
\begin{equation}\label{202507041505}\displaystyle\mathop{\operatorname{ess\,sup}}_{t \in [0,T]}\int_0^t(t-s)^{\alpha-1}\|f(s)\|^2_{L^2(\Omega)}\,ds<\infty,\end{equation}
and $u_0\in L^2(\Omega)$. Let us build a weak solution to the fractional parabolic problem
\begin{equation}\label{202507231532}
\begin{cases}
cD_t^\alpha u(x,t) - \Delta u(x,t) = f(x,t), & (x,t)\in \Omega\times(0,T], \\
u(x,t) = 0, & (x,t)\in \partial \Omega \times [0,T], \\
u(x,0) = u_0(x), & x\in\Omega.
\end{cases}
\end{equation}

Let us begin with the definition of a weak solution to problem \eqref{202507231532}. In this work, we follow exactly the same classical definition of a weak solution for the heat equation, in order to keep our focus on the main subject of this study, namely the fractional energy equations. However, we shall broaden this discussion in future works.

\begin{definition}\label{202507241517}
We say that a function $u \in L^2(0,T; H_0^1(\Omega))$ with $D_t^\alpha \big(u(\cdot)-u_0\big) \in L^2(0,T; H^{-1}(\Omega))$ is a weak solution of \eqref{202507231532} if:
\begin{itemize}
  \item[(i)] for each $v \in H_0^1(\Omega)$ and for almost every $t\in[0,T]$, we have
  $$\big\langle D_t^\alpha(u(t)-u_0), v\big\rangle_{H^{-1}(\Omega),H^1_0(\Omega)} + \big(\nabla u(t), \nabla v\big)_{L^2(\Omega)} = \big(f(t), v\big)_{L^2(\Omega)};$$

  \item[(ii)] the initial condition is satisfied, i.e., $u(0) = u_0$.
\end{itemize}
\end{definition}

\begin{remark}\label{202507271047} We now turn to three topics previously introduced in this section that may not have been entirely clear. Although they are related, each requires a separate discussion.
\begin{itemize}
\item[(i)] First, observe that the regularity condition in \eqref{202507041505}, imposed on function $f$, is more restrictive than the classical assumption $f \in L^2(0, T; L^2(\Omega))$. Indeed, Hölder's inequality yields
\begin{multline*}\qquad\int_0^t\|f(s)\|^2_{L^2(\Omega)}\,ds=\int_0^t(t-s)^{1-\alpha}\Big[(t-s)^{\alpha-1}\|f(s)\|^2_{L^2(\Omega)}\Big]\,ds\\\leq \big[T^{2-\alpha}/(2-\alpha)\big]\int_0^t(t-s)^{\alpha-1}\|f(s)\|^2_{L^2(\Omega)}\,ds.\end{multline*}
This additional regularity is essential in the present proof to obtain that $u \in L^\infty(0, T; L^2(\Omega))$, a key property to use Theorem \ref{teodesdrl}.
\item[(ii)] We also require that $1/2 < \alpha < 1$ to ensure the continuity of the solution from $[0,T]$ into $L^2(\Omega)$. This restriction arises from the limited regularity properties of the Riemann--Liouville fractional integral; see Theorems~\ref{202507221156} and~\ref{202507231029} for details.
\item[(iii)] Finally, we need to discuss condition~$(ii)$ in Definition~\ref{202507241517}, since it only satisfies $u \in L^2(0,T; H_0^1(\Omega))$ and $D_t^\alpha (u(\cdot)-u_0) \in L^2(0,T; H^{-1}(\Omega))$, which is not sufficient to rigorously assign meaning to the initial condition.

To ensure the well-posedness of this requirement, we also need the assumption $1/2 < \alpha < 1$. In fact, as will be shown below, we prove that $D_t^\alpha \big(u(\cdot)-u_0\big) \in L^2(0,T; L^2(\Omega))$. Then, combining this with the regularity assumed in item~$(i)$ above, i.e., $u \in L^\infty(0, T; L^2(\Omega))$, we may apply items~$(i)$ and~$(ii)$ of Proposition~\ref{202507221111} to deduce that
$$
u(t) = u_0 + J_t^\alpha\Big[D_t^\alpha\big(u(t)-u_0\big)\Big],
$$
for almost every $t \in [0,T]$. Moreover, since Proposition~\ref{202507221156} ensures $J_t^\alpha\big[D_t^\alpha(u(t)-u_0)\big]$ belongs to $C([0,T]; L^2(\Omega))$, it follows that $u$ also belongs to $C([0,T]; L^2(\Omega))$. Hence, we may conclude that the initial condition is well-posed.
\end{itemize}
\end{remark}

Classically, we begin by considering $(w_j)_{j \in \mathbb{N}}$ as an orthogonal basis of $H_0^1(\Omega)$ that is also an orthonormal basis of $L^2(\Omega)$ (the complete set of normalized eigenfunctions of the operator $-\Delta$ in $H_0^1(\Omega)$; see \cite[ Section 6.5.1]{Ev01}).

Fix now $m \in \mathbb{N}$. Let us prove the existence of $\{g_{im}\}_{i=1}^{m}$ a sequence of functions, where $g_{im} : [0,T] \to \mathbb{R}$, such that, by defining
\begin{equation}\label{202507231609}
	u_{m}(t):=\sum\limits_{i=1}^{m}g_{im}(t)w_{i},\quad \forall i\in\{1,\ldots,m\},
\end{equation}
we have $u_m(t)$ satisfies
\begin{equation}\label{202507231610}
	\left\{\begin{aligned}
		cD_{t}^{\alpha}(u_{m}(t),w_{j})_{L^2(\Omega)}+(\nabla u_{m}(t),\nabla w_{j})_{L^2(\Omega)} &=(f(t),w_{j})_{L^2(\Omega)},\\ 
		u_{m}(0) &= u_{0m},
	\end{aligned}\right.
\end{equation}
for every $j\in\{1,\ldots,m\}$ and $t\in(0,T]$, where $u_{0m}$ is the orthogonal projection of $u_{0}$ onto $\text{span}\{w_1,\ldots,w_m\}$, given by
\begin{equation}\label{202507251709}
	u_{0m}=\sum\limits_{j=1}^{m}(u_{0},w_{j})_{L^2(\Omega)}w_{j} \rightarrow u_{0}, \quad \text{in } L^2(\Omega), \quad \text{as } m\to\infty.
\end{equation}

From \eqref{202507231609} and \eqref{202507231610}, we conclude that the $(g_{1m}, \ldots, g_{mm})$ must satisfy the following ODE:
\begin{equation}\label{202507241008}\left\{\begin{aligned}
	cD_{t}^{\alpha}U(t)+ B_{m}U(t)&=F_{m}(t),\\
	U(0)&=U_{0m},
\end{aligned}\right.\end{equation}
with
\begin{align*}
	B_{m}=\left[\begin{array}{ccc}
		(\nabla w_{1},\nabla w_{1})_{L^2(\Omega)}&\ldots  & 0 \\
		\vdots       &\ddots  & \vdots\\
		0 &\ldots  & (\nabla w_{m},\nabla w_{m})_{L^2(\Omega)}
	\end{array}\right],
\end{align*}

\begin{align*}
F_{m}(t)=\left[\begin{array}{c}
		( f(t),w_{1})_{L^2(\Omega)}\\
		\vdots  \\
		( f(t),w_{m} )_{L^2(\Omega)}
	\end{array}\right]\hspace*{1cm}\textrm{ and}\hspace*{1cm}
	U_{0m}=\left[\begin{array}{c}
		(u_{0},w_{1})_{L^2(\Omega)}\\
		\vdots  \\
		(u_{0},w_{m})_{L^2(\Omega)}
	\end{array}\right].
\end{align*}

Since $F_m \in L^2(0,T;\mathbb{R}^m)$ (recall item $(i)$ of Remark \ref{202507271047}), \eqref{202507241008} is a system of fractional ordinary differential equations with a forcing term given by a Carathéodory function, which has already been studied in \cite{CarFrTo1,LaWe1}. Therefore, the existence of $g_{jm}\in C([0,T];\mathbb{R})$ with $cD_t^\alpha g_{jm}\in L^2(0,T;\mathbb{R})$, for each $j\in\{1,\ldots,m\}$, is guaranteed.

\subsection{Energy Estimates}

Here, we need to carry out two distinct constructions in order to derive the energy estimates required to prove the existence and uniqueness of solutions to problem \eqref{202507231610}.
\begin{itemize}
\item[(i)] If we multiply the first equation in \eqref{202507231610} by $g_{jm}$ for each $1 \leq j \leq m$ and sum the resulting equations, we obtain
\begin{equation*}
	(cD_{t}^{\alpha}u_{m}(t),u_{m}(t))_{L^2(\Omega)} + (\nabla u_{m}(t),\nabla u_{m}(t))_{L^2(\Omega)} = (f(t),u_{m}(t))_{L^2(\Omega)},
\end{equation*}
for almost every $t \in [0,T]$. Since $u_m\in C([0,T]; L^2(\Omega))$ and we also have $cD_t^\alpha u_m\in L^2(0,T; L^2(\Omega))$,  Corollary \ref{teodesdcap} ensures $cD_{t}^{\alpha}\|u_{m}\|^2_{L^2(\Omega)}\in L^1(0,T;\mathbb{R})$ and
\begin{equation*}
	(1/2)cD_{t}^{\alpha}\|u_{m}(t)\|_{L^2(\Omega)}^{2} + \|\nabla u_{m}(t)\|_{L^2(\Omega)}^{2} \leq (f(t),u_{m}(t))_{L^2(\Omega)},
\end{equation*}
for almost every $t \in [0,T]$. This, together with Young's inequality, implies that there exist constants $c_1, c_2 > 0$ such that
\begin{equation}\label{202507241107}
	D_{t}^{\alpha}\Big(\|u_{m}(t)\|_{L^2(\Omega)}^{2}-\|u_{0m}\|_{L^2(\Omega)}^{2}\Big) + c_1\|u_{m}(t)\|_{H^1_0(\Omega)}^{2} \leq c_2\|f(t)\|^2_{L^2(\Omega)},
\end{equation}
for almost every $t \in [0,T]$.

\begin{itemize}
\item[(a)] By integrating both sides of \eqref{202507241107} from $0$ to $t\in(0,T]$, thanks to item $(i)$ of Proposition \ref{202507221111}, we obtain
\begin{multline*}
	\qquad\qquad\quad\dfrac{1}{\Gamma(1-\alpha)}\int_0^t(t-s)^{-\alpha}\Big(\|u_{m}(s)\|_{L^2(\Omega)}^{2}-\|u_{0m}\|_{L^2(\Omega)}^{2}\Big)\,ds\\+c_1\int_0^t\|u_{m}(s)\|_{H^1_0(\Omega)}^{2}\,ds\leq c_2\int_0^t\|f(s)\|^2_{L^2(\Omega)}\,ds,
\end{multline*}
what implies that there exists $C_1>0$ such that
\begin{equation}\label{202507241059}
\| u_m \|_{L^2(0,T; H_0^1(\Omega))}^2 \leq C_1 \left( \| u_0 \|_{L^2(\Omega)}^2 + \| f \|_{L^2(0,T; L^2(\Omega))}^2 \right).
\end{equation}
For the fact that $\| f \|_{L^2(0,T; L^2(\Omega))}$ is finite, see item $(i)$ of Remark \ref{202507271047}.

\item[(b)] We now proceed in a standard way. Fix $v \in H_0^1(\Omega)$ with $\|v\|_{H_0^1(\Omega)} \leq 1$, and decompose $v = v_1 + v_2$, where $v_1 \in \text{span}\{w_1, \ldots, w_m\}$ and $(v_2, w_k)_{L^2(\Omega)} = 0$ for every $k \in \{1, \ldots, m\}$. Since functions $(w_j)_{j \in \mathbb{N}}$ are orthogonal in $H_0^1(\Omega)$, it follows that
\begin{multline*}\qquad\qquad\quad(v_2,w_k)_{H^1_0(\Omega)}=\left(\sum_{j=1}^\infty(v_2,w_j)_{L^2(\Omega)}w_j,w_k\right)_{H^1_0(\Omega)}\\=(v_2,w_k)_{L^2(\Omega)}(w_k,w_k)_{H^1_0(\Omega)}=0,\end{multline*}
for each $k\in\{1,\ldots,m\}$, and therefore
$$\|v_1\|_{H_0^1(\Omega)} \leq \|v\|_{H_0^1(\Omega)} \leq 1.$$

Then, from the first equation in \eqref{202507231610}, since $(w_j)_{j \in \mathbb{N}}$ are orthonormal in $L^2(\Omega)$, we obtain
\begin{multline*}
\qquad\qquad\quad(cD_t^\alpha u_m(t), v)_{L^2(\Omega)} = (cD_t^\alpha u_m(t), v_1)_{L^2(\Omega)} \\
= (f(t), v_1)_{L^2(\Omega)} - (\nabla u_m(t), \nabla v_1)_{L^2(\Omega)},
\end{multline*}
for almost every $t \in [0,T]$. Combined with \eqref{202507241059}, this implies the existence of a constant $C_2 > 0$ such that
\begin{equation}\label{202507241100}
\|cD_t^\alpha u_m\|^2_{L^2(0,T;H^{-1}(\Omega))} \leq C_2 \left( \| u_0 \|_{L^2(\Omega)}^2 + \| f \|_{L^2(0,T; L^2(\Omega))}^2 \right).
\end{equation}

\item[(c)] Finally, by applying $J_t^\alpha$ to both sides of \eqref{202507241107} and using items $(i)$ and $(ii)$ of Proposition \ref{202507221111}, we obtain
\begin{multline*}
	\qquad\qquad\quad\|u_{m}(t)\|_{L^2(\Omega)}^{2} + c_1J_t^\alpha\Big(\|u_{m}(t)\|_{H^1_0(\Omega)}^{2}\Big) \leq \|u_{0m}\|_{L^2(\Omega)}^{2}+ c_2J_t^\alpha\Big(\|f(t)\|^2_{L^2(\Omega)}\Big),
\end{multline*}
for almost every $t \in [0,T]$, which leads to the existence of a constant $C_3 > 0$ such that
\begin{multline}\label{202507271006}
	\qquad\qquad\quad\|u_{m}(t)\|_{L^\infty(0,T;L^2(\Omega))}^{2}\leq C_3\left(\|u_{0}\|_{L^2(\Omega)}^{2}\right. \\+ \left.\displaystyle\mathop{\operatorname{ess\,sup}}_{t \in [0,T]}\int_0^t(t-s)^{\alpha-1}\|f(s)\|^2_{L^2(\Omega)}\,ds\right).
\end{multline}
\end{itemize}
\item[(ii)]We now need to perform an estimate that is not typically used in the classical analysis of the heat equation, but which is of fundamental importance in our setting. If we multiply the first equation in \eqref{202507231610} by $cD_t^\alpha g_{jm}$ for each $1 \leq j \leq m$ and sum the resulting equations, we obtain
\begin{multline*}
	\qquad\quad(cD_{t}^{\alpha}u_{m}(t),cD_{t}^{\alpha}u_{m}(t))_{L^2(\Omega)} + \big(\nabla u_{m}(t),\nabla \big[cD_{t}^{\alpha}u_{m}(t)\big]\big)_{L^2(\Omega)} \\= (f(t),cD_{t}^{\alpha}u_{m}(t))_{L^2(\Omega)},
\end{multline*}
for almost every $t \in [0,T]$. Since $\nabla u_m \in C([0,T]; L^2(\Omega))$ and $cD_t^\alpha \nabla u_m \in L^2(0,T; L^2(\Omega))$, Corollary \ref{teodesdcap}, together with Young's inequality, implies that $cD_{t}^{\alpha}\|\nabla u_{m}\|^2_{L^2(\Omega)} \in L^1(0,T;\mathbb{R})$ and 
\begin{equation}\label{202507271058}
	\|D_{t}^{\alpha}u_{m}(t)\|_{L^2(\Omega)}^{2} + cD_t^\alpha\|\nabla u_{m}(t)\|_{L^2(\Omega)}^{2} \leq \|f(t)\|^2_{L^2(\Omega)},
\end{equation}
for almost every $t \in [0,T]$. By integrating both sides of \eqref{202507271058} from $0$ to $t\in(0,T]$, we obtain
\begin{multline*}
	\qquad\quad\int_0^t\|cD_t^\alpha u_{m}(s)\|_{L^2(\Omega)}^{2}\,ds\\+\dfrac{1}{\Gamma(1-\alpha)}\int_0^t(t-s)^{-\alpha}\|u_{m}(s)\|_{H^1_0(\Omega)}^{2}\,ds\leq \int_0^t\|f(s)\|^2_{L^2(\Omega)}\,ds,
\end{multline*}
what implies that 
\begin{equation}\label{202507271027}
\| cD_t^\alpha u_m \|_{L^2(0,T; L^2(\Omega))}^2 \leq \| f \|_{L^2(0,T; L^2(\Omega))}^2.
\end{equation}
\end{itemize}

\subsection{Existence, Uniqueness, and Some Regularity Properties of the Solution}

We now aim to prove that the sequence $(u_m)_{m \in \mathbb{N}}$ admits a subsequence that converges weakly to a function, which we shall show to be the weak solution of \eqref{202507231532}. For clarity and organization, we state this result explicitly below.

\begin{theorem} There exists a unique weak solution $u$ to problem \eqref{202507231532}. Moreover, $u \in C([0,T]; L^2(\Omega))$, with $cD_t^\alpha u \in L^2(0,T; L^2(\Omega))$.
\end{theorem}

\begin{proof} We shall split this proof in 4 steps.\vspace*{0.2cm}

\noindent\textbf{Step 1: The Weak Convergences.}

From the energy estimates \eqref{202507241059}, \eqref{202507241100}, \eqref{202507271006} and \eqref{202507271027}, by the Banach--Alaoglu Theorem, there exists a subsequence $(u_{m_l})_{l \in \mathbb{N}}\subset (u_m)_{m \in \mathbb{N}}$ and a function $u \in L^2(0,T; H_0^1(\Omega)) \cap L^\infty(0,T; L^2(\Omega))$ such that $D_t^\alpha(u(\cdot)-u_0) \in L^2(0,T; L^2(\Omega))$ and
\begin{equation} \label{202507251621}
\left\{
\begin{array}{rll}
u_{m_l} &\rightharpoonup u,& \text{in } L^2(0,T; H_0^1(\Omega)), \\
u_{m_l} &\overset{*}{\rightharpoonup} u,& \text{in } L^\infty(0,T; L^2(\Omega)),\\
D_t^\alpha (u_{m_l}-u_{0{m_l}}) &\rightharpoonup D_t^\alpha (u-u_{0}),& \text{in } L^2(0,T; L^2(\Omega)).\vspace*{0.2cm}
\end{array}
\right.
\end{equation}

\noindent\textbf{Step 2: Proving that $u$ is the Weak Solution.} 

Fix an integer $N$ and consider a test function $v \in C^\infty_c((0,T); H_0^1(\Omega))$ of the form
\begin{equation} \label{testv}
\nu(t) = \sum_{k=1}^N d^k(t) w_k,
\end{equation}
where each $d^k(t)$ belongs to $C^\infty_c((0,T);\mathbb{R})$. Choosing $m \geq N$, we multiply the first equation in \eqref{202507231610} by $d^k(t)$ for $k = 1,\dots, N$, and integrate over $[0,T]$ to obtain
\begin{multline*}
\int_0^T ( D_t^\alpha(u_m(t)-u_{0m}), \nu(t) )_{L^2(\Omega)}\, dt +\int_0^T(\nabla u_m(t), \nabla \nu(t))_{H^1_0(\Omega)} \, dt \\= \int_0^T (f(t), \nu(t))_{L^2(\Omega)} \, dt.
\end{multline*}

Now, by taking $m = m_l$ and using the convergences \eqref{202507251709} and \eqref{202507251621}, we can pass to the limit and derive
\begin{multline} \label{intformu}
\int_0^T ( D_t^\alpha(u(t)-u_0), \nu(t) )_{L^2(\Omega)}\, dt +\int_0^T(\nabla u(t), \nabla \nu(t))_{H^1_0(\Omega)} \, dt \\= \int_0^T (f(t), \nu(t))_{L^2(\Omega)} \, dt.
\end{multline}

Since functions of the form \eqref{testv} are dense in $L^2(0,T; H_0^1(\Omega))$, we conclude that \eqref{intformu} holds for every $\nu$ in this space. In particular, for almost every $t \in [0,T]$ and every $v \in H_0^1(\Omega)$, it holds that
\begin{equation} \label{202507271140}\big\langle D_t^\alpha(u(t)-u_0), v\big\rangle_{H^{-1}(\Omega),H^1_0(\Omega)} + \big(\nabla u(t), \nabla v\big)_{L^2(\Omega)} = \big(f(t), v\big)_{L^2(\Omega)}.\vspace*{0.2cm}\end{equation}
 
\noindent\textbf{Step 3: Regularity of $u$ and the Initial Condition.} 

Here we simply observe that, by using items $(i)$ and $(ii)$ of Proposition \ref{202507221111} and the regularities of $u$ given in \eqref{202507251621}, we obtain
$$J^\alpha_t \big[D^\alpha_t \big(u(t)-u_0)\big] = u(t)-u_0,$$
for almost every $t \in [0,T]$. Then, by applying Proposition \ref{202507221156} and noting that $1/2 < \alpha < 1$, we conclude that $u \in C([0,T];L^2(\Omega))$, and finally, thanks again to item $(i)$ of Proposition \ref{202507221111}, that $u(0) = u_0$.
\vspace*{0.3cm}

\noindent\textbf{Step 4: Uniqueness of Solution.} 

Consider a pair of weak solutions $u_1$ and $u_2$ satisfying the properties established so far. Define $\sigma = u_1 - u_2$, and observe that, thanks to \eqref{202507271140}, we have:
\begin{equation*}
( D_t^\alpha\sigma(t), \nu(t) )_{L^2(\Omega)} +(\nabla \sigma(t), \nabla \nu(t))_{H^1_0(\Omega)} = 0
\end{equation*}
for almost every $t\in[0,T]$. 

Now, given that $\sigma \in L^2(0,T; H_0^1(\Omega))$, we set $\nu(t) = \sigma(t)$. Since we already know that $\sigma \in C([0,T]; L^2(\Omega))$ and $D_t^\alpha \sigma \in L^2(0,T; L^2(\Omega))$, we obtain from Theorem \ref{teodesdrl} that $D_t^\alpha \|\sigma\|^2$ belongs to $L^1(0,T;\mathbb{R})$ and
\begin{equation}\label{202507271121}
D_t^\alpha\left\|\sigma(t)\right\|^2_{L^2(\Omega)}\leq0,
\end{equation}
for almost every $t\in[0,T]$. Finally, item $(ii)$ of Proposition \ref{202507221111} ensures that, by applying $J_t^\alpha$ to both sides of \eqref{202507271121}, we obtain
\begin{equation*}
\left\|\sigma(t)\right\|^2_{L^2(\Omega)} \leq 0,
\end{equation*}
for almost every $t \in [0,T]$, which leads us to conclude that $\sigma(t) \equiv 0$ in $L^2(\Omega)$ for almost every $t \in [0,T]$, thus proving uniqueness.
\end{proof}

\section*{Acknowledgements}
The authors Juan C. O. Ballesteros and Pedro G. P. Torelli gratefully acknowledge the support of the \textit{Coordenação de Aperfeiçoamento de Pessoal de Nível Superior} – Brazil (CAPES), through a doctoral scholarship received during the development of this work.

\end{document}